\pdfoutput=1 

\RequirePackage{fix-cm}
\documentclass[smallextended]{svjour3}
\usepackage[utf8]{inputenc}
\usepackage{booktabs}
\usepackage{amssymb,amsmath,epsfig,graphics,psfrag,graphicx,color,subfigure}
\usepackage{algorithm} 
\usepackage{algpseudocode} 
\usepackage[title]{appendix}
\renewenvironment{proof}{\noindent{\it Proof.}}{\hfill$\square$}
\definecolor{lightblue}{rgb}{0.22,0.45,0.70}
\definecolor{lightgreen}{rgb}{0.22,0.50,0.25}
\usepackage[right = 2.5cm, left=2.5cm, top = 2.5cm, bottom =2.5cm]{geometry}
\usepackage[colorlinks=true,breaklinks=true,linkcolor=lightblue,citecolor=lightblue]{hyperref}

\definecolor{darkred}{rgb}{0.82,0.15,0.20}
\definecolor{darkblue}{rgb}{0.82,0.15,0.12}

\usepackage{mathtools}
\DeclarePairedDelimiterX{\triplenorm}[1]{\vert\Vert}{\Vert\vert}{#1}

\numberwithin{equation}{section}
\numberwithin{figure}{section}
\numberwithin{table}{section}

\newcommand{\buS}{\underline{\boldsymbol{S}}}

\newcommand\cero{{\boldsymbol{0}}}

\newcommand\bI{\mathbf{I}}

\newcommand\bV{\mathbf{V}}

\newcommand\bPi{\boldsymbol{\Pi}}

\newcommand\ff{\boldsymbol{f}}
\newcommand\bg{\boldsymbol{g}}

\newcommand\beps{\boldsymbol{\varepsilon}}
\newcommand\nn{{\boldsymbol{n}}}

\newcommand{\tP}{\tilde{P}}

\newcommand\buu{\underline{\boldsymbol{u}}}
\newcommand\buv{\underline{\boldsymbol{v}}}
\newcommand\buw{\underline{\boldsymbol{w}}}
\newcommand\buz{\underline{\boldsymbol{z}}}

\newcommand\bx{\boldsymbol{x}}
\newcommand\by{\boldsymbol{y}}
\newcommand\bu{\boldsymbol{u}}
\newcommand\bv{\boldsymbol{v}}

\newcommand\bz{\boldsymbol{z}}

\newcommand{\bYd}{\boldsymbol{Y}_d}
\newcommand{\bY}{\boldsymbol{Y}}

\newcommand\cT{\mathcal{T}}

\newcommand{\bbV}{\mathbb{V}}
\newcommand{\bbR}{\mathbb{R}}
\newcommand{\bbP}{\mathbb{P}}

\newcommand\bdiv{{\mathop{\mathbf{div}}\nolimits}}
\newcommand\vdiv{{\mathop{\mathrm{div}}\nolimits}}

\newcommand\bDelta{\boldsymbol{\Delta}}

\newcommand{\dx}{\,\mathrm{d}x}

\newcommand{\ds}{\,\mathrm{d}s}

\allowdisplaybreaks

\begin{document}
	\titlerunning{VEM for optimal control problems in poroelasticity equation}
	\title{Virtual element approximations for distributed or Neumann boundary optimal control problems governed by poroelasticity equation}
	\authorrunning{Gudi, Mora \& Verma}
	\author{Thirupathi Gudi, David Mora, and Nitesh Verma}
	\institute{
		Thirupathi Gudi \at 
		Department of Mathematics, Indian Institute of Science, Bangalore, India.\\
		\email{gudi@iisc.ac.in}.
		\and
		David Mora \at GIMNAP, Departamento de Matem\'atica, Universidad del B\'io-B\'io, Concepci\'on, Chile; and CI$^2$MA,
		Universidad de Concepci\' on, Chile.\\
		\email{dmora@ubiobio.cl}.
		\and
		Nitesh Verma  (corresponding author) \at
		GIMNAP, Departamento de Matem\'atica, Universidad del B\'io-B\'io, Concepci\'on, Chile.\\
		\email{nverma@ubiobio.cl}
	}
	\date{\today}
	\maketitle
	
	\begin{abstract}
        This paper investigate the conforming virtual element method for the optimal control problem governed by the linear poroelasticity equation. Both distributed control and Neumann boundary controls are considered with the optimize-then-discretize approach. We establish the well-posedness of the discrete state and adjoint problems in three-field formulation by deriving the first-order optimality condition. The optimal a priori error estimates are developed through a novel poroelastic projection for both control problems, and these estimates are uniform with respect to the relevant physical parameters of the model problem. Moreover, the numerical experiments are carried out using the primal dual active set strategy, and demonstrate the support for the derived theoretical results.
	\end{abstract}
	

	\section{Introduction}

    The study of fluid flow through deformable porous media originates with the seminal work of Terzaghi in one dimension \cite{terzhagi25}, later generalized to three dimensions by Biot \cite{biot41}. The resulting system, now known as Biot’s equations of poroelasticity, provides a rigorous framework for modeling the coupled interaction between solid deformation and fluid transport. Poroelasticity has become indispensable in diverse scientific and engineering applications, including geoscience (e.g., swelling in coal seams, soil consolidation), biomedical science (e.g., tissue growth, glaucoma modeling), and industrial processes such as absorbent hygiene products \cite{collis17,nespoli25}. Over the decades, multiple formulations and analytical approaches have been proposed to study poroelasticity \cite{or_sinum16,yi17,phillips,verma24,verma26}, reflecting both its mathematical richness and practical importance. 

    Polygonal discretizations have gained increasing attention in recent years due to their ability to handle complex geometries, irregular meshes, and hanging-node refinements. Among these, the Virtual Element Method (VEM) \cite{ahmad13,beirao-b13,beirao-div17} has emerged as a powerful extension of the finite element method to general polygonal meshes. VEM offers several advantages: it provides stability in mixed formulations, avoids locking in nearly incompressible regimes, and allows for flexible mesh generation in complex domains. These features make VEM particularly well-suited for poroelasticity, where mixed formulations and geometric complexity are inherent.
    
    Optimal control problems governed by poroelasticity have recently emerged as a significant research direction. Such problems are motivated by the need to regulate displacement, pressure, or flux in applications ranging from medical science to subsurface engineering. Bociu et al. \cite{bociu22} established well-posedness results for poroelasticity control problems, with particular emphasis on biomedical applications such as glaucoma \cite{bociu22,causin14}. Related studies include land subsidence modeling \cite{xu12}, further underscoring the relevance of poroelastic control in geomechanics. More broadly, optimal control of elliptic and parabolic PDEs has been extensively studied in \cite{troltzsch10,manzoni21}, providing a theoretical foundation for poroelasticity control. One of the ways the discrete optimality system is implemented via the Primal–Dual Active Set (PDAS) strategy \cite{troltzsch10}. Within the PDAS framework, the active set is updated iteratively by checking the control constraints against the projection formula, while the state and adjoint equations are solved on the inactive set. This iterative procedure guarantees finite termination and efficient enforcement of admissible bounds. Recent numerical work has focused on finite element approximations for distributed control in unsteady quasi-static poroelasticity \cite{harpal26}, where the control acts as a source or sink of fluid. However, these studies have been limited to distributed control, and no prior work has addressed Neumann boundary control or the use of VEMs in this context. Moreover, the present work considers control applied directly to the displacement variable in contrast and hence referring to different class of optimization problems. This model can also be extended to incorporate interaction with fluid such as Biot-Brinkman \cite{ricardo24,verma26}, Biot-Stokes equations \cite{ricardo22} among others.
    
    In this paper, we present the first comprehensive study of conforming VE approximations for optimal control problems governed by poroelasticity equations with distributed and Neumann boundary controls. We adopt the three-field formulation involving displacement, pressure, and total pressure, and follow the optimize-then-discretize paradigm. The lowest-order conforming VE spaces are employed for the state and adjoint problems, while the control variable is discretized using piecewise constants. The discrete optimality system is implemented via the PDAS strategy, which ensures efficient enforcement of control constraints and finite termination. To establish rigorous convergence results, we introduce novel poroelastic projection operators that enable unified error analysis for both distributed and Neumann boundary control problems.

\noindent The main contributions of this work can be summarized as follows:
\begin{itemize}
    \item {\it Novel formulation}: We provide the first study of distributed and Neumann boundary control problems for poroelasticity in the three-field setting, including derivation of the state, adjoint, and optimality systems.
    
    \item {\it Well-posedness analysis}: We establish the well-posedness of both the continuous and discrete formulations, showing their equivalence with energy minimization principles.
    
    \item {\it Unified error estimates}: We develop projection-based rigorous error estimates for both control types, introducing novel poroelastic projection operators for the state and adjoint problems. These estimates are uniform with respect to the physical parameters.
    
    \item {\it Low-order VEM discretization}: We employ lowest‑order conforming VEM spaces, specifically designed to accommodate the minimal regularity of the control variable, thereby ensuring robustness of the scheme and efficiency of the implementation.
    
    \item {\it Algorithmic implementation}: We integrate the PDAS strategy into the VEM framework, demonstrating its effectiveness through numerical experiments for both distributed and Neumann boundary control.
\end{itemize}

The remainder of the paper is organized as follows. Section~\ref{sec:model} introduces the governing equations, weak formulation, and well-posedness of the continuous optimal control problem. Section~\ref{sec:abstract} develops the abstract error analysis using projection operators. Section~\ref{sec:VEapprox} presents the conforming VEM discretization and establishes error estimates in the energy norm. Section~\ref{sec:numer} describes the implementation of the PDAS algorithm and provides numerical experiments for both distributed and Neumann boundary control problems.

	\section{Governing equations}\label{sec:model}
	
%

We focus here on the quasi-static and steady state poroelasticity equation for simplicity, and introduce a variable known as the total pressure \cite{or_sinum16}. The governing equation in three-field form
for a given body force $\ff$ and source/sink $\ell$ in domain $\Omega$ reads, in terms of solid displacement $\bu$ of the porous medium, fluid pressure $p$ and total pressure $\psi$, as: 
	\begin{subequations}\label{eq:strong-poro}
		\begin{align}
			-\bdiv(2{\mu} \beps(\bu)  -\psi \bI)  &= \ff, \label{eq:momentum}\\
			\left(c_0 + \frac{\alpha^2}{\lambda} \right) p - \frac{\alpha}{\lambda} \psi - \frac{1}{\eta} \vdiv(\kappa  \nabla p ) &= \ell, \label{eq:massconservation}\\
			 -\vdiv \bu + \frac{\alpha}{\lambda} p - \frac{1}{\lambda} \psi &=0, \label{eq:totalpres} 
		\end{align}
	\end{subequations}
	with the boundary conditions
	\begin{subequations}\label{eq:bc}
		\begin{align}
			\bu = \cero \quad \text{and} \quad \frac{\kappa}{\eta} \nabla p \cdot\nn^{\partial \Omega} = 0  & &\text{on $\Gamma_D$}, \label{bc:Dirichlet}\\
			(2{\mu} \beps(\bu)  -\psi \bI) \nn^{\partial \Omega} = \bg \quad  \text{and}\quad p=0
			& &\text{on $\Gamma_N$}, \label{bc:Neumann}
		\end{align}
	\end{subequations}
	where the boundary $\Gamma= \partial \Omega$ is split as $\Gamma:= \Gamma_D \cup\Gamma_N$, $\Gamma_D \cap\Gamma_N = \emptyset$ ($\Gamma_D$: Dirichlet boundary part and $\Gamma_N$: Neumann boundary part for displacement), $\nn^{\partial \Omega}$ is the unit normal on boundary $\partial \Omega$, and $\kappa(x)$ is the hydraulic conductivity of the porous medium (assuming $0<\kappa_1 \le \kappa(x) \le \kappa_2 < \infty$), 
    $\eta$ is the constant viscosity of the interstitial fluid, $c_0$ is the constrained specific storage coefficient, 
    $\alpha$ is the Biot-Willis consolidation parameter, 
    and $\mu$ and $\lambda$ are the shear and dilation moduli associated with the constitutive law of the solid structure.

	Here the distributed control stated as the problem \eqref{eq:strong-poro} with $\ff+ \by$ (where $\by \in \bY:= [L^2(\Omega)]^2$ ) instead of $\ff$ in equation \eqref{eq:momentum} and boundary conditions \eqref{eq:bc}; Whereas the Neumann control problem stated as the problem \eqref{eq:strong-poro} and boundary condition \eqref{eq:bc} with $\bg +\by$ (where $\by \in \bY:= [L^2(\Gamma_N)]^2$ ) instead of $\bg$ in equation $\eqref{bc:Neumann}$. Note that we will consider homogeneous boundary conditions ($\bg=\cero$) throughout the analysis for simplification.
	
	Introducing the Sobolev space
	$H^1_*(\Omega):= \{ v \in H^1(\Omega): v =0 \text{ a.e. on } \Gamma_*\},$ for $*=\{D,N\}$ and define the spaces
	\begin{align*}
		&\bV := [H^1_D(\Omega)]^2, \quad Q := H^1_N(\Omega), \quad Z:= L^2(\Omega), \\
		&\text{and} \quad \bYd:= \{ \by \in \bY: \by_a \le \by(x) \le \by_b, \text{ for a.e. } x \in \Omega \text{, or } x \in\Gamma_N\}, \,\,  \by_a, \by_b \in \bbR^2, \,\by_a< \by_b,
	\end{align*}
	where the bounds in $\bYd$ are in the componentwise sense.

    \noindent Throughout this paper, we adopt the standard notation of norms as, 
    \begin{gather*}
        \|v\|_{0, \Omega}^2 := \int_{\Omega} |v|^2 \dx \quad \forall v \in L^2(\Omega), \qquad \| v\|_{1, \Omega}^2:= \|v\|_{0, \Omega}^2 + \| \nabla v\|_{0, \Omega}^2 \quad \forall v \in H^1(\Omega), \\
        \quad \text{ and } \quad
        \| \by\|_{\bY}^2:= \langle \by, \by \rangle_{\bY} \quad \forall \by \in \bY.
    \end{gather*}
    With a slight abuse of notation, the same symbol will be used for norms in vector-valued spaces $[L^2(\Omega)]^2$ and  $[H^1(\Omega)]^2$.
    
	
	The weak formulation of  \eqref{eq:strong-poro} for a given distributed, or Neumann control $\by \in \bYd$ 
	 can be stated as: find $\buu:=(\bu, p, \psi) \in \bbV:= \bV \times Q \times Z$
	 such that
	\begin{subequations} \label{eq:weak-control}
		\begin{align}
			2 \mu \int_{\Omega} \beps(\bu):\beps(\bv) - \int_{\Omega}\psi \, \vdiv \bv & = \int_{\Omega} \ff \cdot \bv 
            + \langle B \bv, \by \rangle_{\bY} 
			 & \quad \forall \bv \in \bV, \\
			\biggl(c_0 +\frac{\alpha^2}{{\lambda}}\biggr) \int_{\Omega} p \,q + \frac{1}{\eta} \int_{\Omega} \kappa \nabla p \cdot \nabla q - \frac{\alpha}{{\lambda}} \int_{\Omega} \psi\, q &= \int_{\Omega}  \ell \, q & \quad \forall q \in Q, \\
			- \int_{\Omega} \phi\, \vdiv \bu  + \frac{\alpha}{\lambda} \int_{\Omega} p \,\phi - \frac{1}{\lambda} \int_{\Omega} \psi \, \phi & = 0 & \quad \forall \phi \in Z,
		\end{align}
	\end{subequations}
	where $B: \bV \to \bY_d$ is a map so that $\langle B \bv, \by \rangle_{\bY}$ is well-defined for both distributed and Neumann controls.\\
		In addition, denoting the bilinear forms and linear functionals as,
	\begin{gather*}
		a_1(\bu, \bv):=2 \mu \int_{\Omega} \beps(\bu):\beps(\bv) \dx, \qquad 
		b_1(\bv, \phi):= - \int_{\Omega} \phi\, \vdiv \bv \dx, \qquad  b_2(q, \phi):= \frac{\alpha}{{\lambda}} \int_{\Omega} \phi \, q \dx, \\
		a_2(p, q):= \int_{\Omega} p \,q \dx + \frac{1}{\eta} \int_{\Omega} \kappa \nabla p \cdot \nabla q \dx, \qquad a_3(\psi, \phi):=  \frac{1}{{\lambda}} \int_{\Omega} \psi \,\phi \dx, \\
		F(\bv):= \int_{\Omega} \ff \cdot  \bv \dx,
        \quad \text{and} \quad L(q):= \int_{\Omega} \ell \, q \dx,
	\end{gather*}
	then variational formulation for a given control $\by \in \bYd$ is rewritten as,
	\begin{subequations} \label{eq:weak-control-state}
		\begin{align}
		a_1(\bu, \bv) + b_1(\bv, \psi) &= F(\bv)+         \langle B \bv, \by \rangle_{\bY }        & \forall \bv \in \bV, \\
		a_2(p, q) - b_2(q, \psi) & = L(q) & \forall q \in Q, \\
		b_1(\bu, \phi) + b_2(p, \phi) - a_3(\psi, \phi) & = 0  & \forall \phi \in Z.
		\end{align}
	\end{subequations}
%

	\noindent We have the following properties of the bilinear forms:
	\begin{itemize}
		\item the bilinear form $a_{1}( \cdot, \cdot): \bV \times \bV \to \bbR$ is bounded and coercive,
			\begin{align*}
				a_{1}( \bu, \bv) & \le 2 \mu \, \| \beps(\bu)\|_{0, \Omega} \| \beps(\bv) \|_{0, \Omega} \le 2 \mu \, C_{k,2} \| \bu \|_{1, \Omega} \| \bv \|_{1, \Omega}, \\
				a_{1}( \bv, \bv) & \ge 2 \mu \, \| \beps(\bv)\|_{0, \Omega}^2 \ge 2 \mu \, C_{k,1} \| \bv \|_{1, \Omega}^2;
			\end{align*}
			
		\item the bilinear form $a_{2}( \cdot, \cdot): Q \times Q \to \bbR$ is bounded and coercive,
		\begin{align*}
			a_{2}( p, q) & \le \left( c_0 + \frac{\alpha^2}{\lambda} \right) \|p\|_{0, \Omega} \|q \|_{0, \Omega} + \frac{\kappa_2}{\eta}  \|\nabla p\|_{0, \Omega} \| \nabla q \|_{0, \Omega} \\
			& \le \max \left\{ \left( c_0 + \frac{\alpha^2}{\lambda} \right),  \frac{\kappa_2}{\eta} \right\} \|p \|_{1, \Omega}\|q \|_{1, \Omega}, \\
			a_{2}( q, q) & \ge \left( c_0 + \frac{\alpha^2}{\lambda} \right) \|q \|_{0, \Omega}^2 + \frac{\kappa_1}{\eta} \| \nabla q \|_{0, \Omega}^2 \ge \min \left\{\left( c_0 + \frac{\alpha^2}{\lambda} \right),  \frac{\kappa_1}{\eta} \right\} \|q \|_{1, \Omega}^2;
		\end{align*}
		
		\item the bilinear form $b_1(\cdot, \cdot)$ satisfies inf-sup condition on $\bV \times Z$, there exists $\beta>0$ such that
		\begin{align*}
			\sup_{(0 \neq) \bv \in \bV} \frac{b_1(\bv, \phi)}{\| \bv \|_{1, \Omega}} \ge \beta \| \phi \|_{0, \Omega} \qquad \forall \phi \in Z;
		\end{align*}
		\item the bilinear forms $b_2(\cdot, \cdot)$ and $a_3(\cdot, \cdot)$ are bounded,
		\begin{align*}
				b_2(q, \phi)&\le \frac{\alpha}{{\lambda}} \|\phi\|_{0, \Omega} \| q \|_{0, \Omega}, \qquad
			a_3(\psi, \phi) \le  \frac{1}{{\lambda}} \|\psi\|_{0, \Omega} \|\phi \|_{0, \Omega};
		\end{align*}
		\item the linear functionals $F(\cdot): \bV \to \bbR$ and $L(\cdot): Q \to \bbR$ are bounded,
		\begin{align*}
			F(\bv) + \langle B \bv, \by \rangle & \le \|\ff \|_{0, \Omega} \|\bv\|_{0, \Omega} 
            + \| B \bv\|_{\bY} \| \by \|_{\bY} 
            \le \big(  \|\ff \|_{0, \Omega} 
            + \| \by \|_{\bY}  \big) \|\bv\|_{1, \Omega},\\
			L(q)& \le \| \ell \|_{0, \Omega}\| q \|_{0, \Omega}.
		\end{align*}
\end{itemize}

	\noindent Building from the above result, we state the following theorem, adapted from \cite{or_sinum16,burger_acom21}, with a proof sketch included for completeness, which establishes the continuity of the control-to-state map.
	\begin{theorem} \label{thm:Gwelldefined}
	Assuming that the $B: \bV \to \bYd$ is a continuous map then the solution map denoted by $\buS: \bYd\to \bbV $ from control variable $\by \in \bYd$ to the state variable $\buu_y \in \bbV$ given by \eqref{eq:weak-control}, is a well-defined and continuous map.
	\end{theorem}
	\begin{proof}
		For given control $\by \in \bYd$, we want to show that the problem \eqref{eq:weak-control} is well-posed, or
        the map $\buS: \by (\in \bYd) \mapsto \buu (\in \bbV)$ is well-defined.
		
		The continuous dependence of the solution $\buu \in \bbV$ on the given load functions $\ff, \ell$ and $\by$ using the properties of bilinear forms can be given as,
		\begin{equation} \label{eq:cts_dependence}
			\begin{aligned}
			2 \mu \, C_{k,1} \| \bu \|_{1, \Omega}^2 &+  c_0 \| p \|_{0, \Omega}^2 + \frac{\kappa_1}{\eta} \| \nabla p \|_{0, \Omega}^2 + \| \psi\|_{0, \Omega}^2 \\
			& \le  2 \left( 1+ \frac{4}{C_{k,1}^2} \left( 1+ \frac{1}{2 \mu} \right)\right) \big( \|\ff \|_{0, \Omega}^2
            + \| \by \|_{\bY}^2 \big)+  \frac{ 4\eta}{\kappa_1} \left( 1+ 2 \mu \right)\| \ell \|_{0, \Omega}^2.
		\end{aligned}
	\end{equation}
		In addition, uniqueness of the solution can be implied by taking $\ff=\cero, \by = \cero$ and $\ell=0$ in \eqref{eq:cts_dependence}. We can obtain the existence of the solution, referring to \cite{or_sinum16}, by approaching to the Fredholm alternative theory for operators due to lack of symmetry in the formulation.
	\end{proof}

Therefore, the distributed, or Neumann boundary optimal control problem is stated as 
\begin{equation}\label{eq:optimal-abstract}
	J(\buu, \by) = \min_{(\buw, \bx) \in \bbV \times \bYd} J(\buw, \bx),
\end{equation}
subject to the constraints \eqref{eq:weak-control-state}, and the quadratic cost functional $J: \bbV \times \bYd \to \bbR$ is given as,
\begin{align*}
	J(\buw, \bx)= \frac{1}{2} \| \buw - \buu_d \|_{0, \Omega}^2 + \frac{\gamma}{2} \| \bx \|_{\bY}^2, \qquad \buw \in \bbV, \, \bx \in \bYd,
\end{align*}
where $\gamma>0$ is the regularization parameter, $\buu_d:=(\bu_d, p_d, \psi_d) $ is the given desired functions and $\bYd$ is the admissible space for control.


We decompose the solution of \eqref{eq:weak-control-state} as $\buu:= \buu_{0}+ \buu_y$ with $\buu_0:=(\bu_0, p_0, \psi_0) \in \bbV$ satisfying the poroelasticity equation \eqref{eq:weak-control} with $\by=\cero$; and  $\buu_y:=(\bu_y, p_y, \psi_y) \in \bbV$ solution of \eqref{eq:weak-control} with $\ff=\cero$.

Now, we define the solution operator $\buS:=(S^u, S^p, S^\psi):\bYd \to \bbV $ of the problem \eqref{eq:weak-control-state} such that 
$$\buS (\bx) = (S^u \bx, S^p \bx, S^{\psi} \bx) :=\buu_x=(\bu_x, p_x, \psi_x),$$ then $\buw=\buu_0 + \buS (\bx)$ from \eqref{eq:weak-control-state} and
we can rewrite the optimal control problem \eqref{eq:optimal-abstract} in terms of only control variable as,
\begin{align} \label{eq:optimalproblem}
	j(\by)= \min_{\bx \in \bYd} j(\bx) := \min_{\bx \in \bYd} \Big(\frac{1}{2} \| \buu_0 + \buS (\bx) - \buu_d \|_{0, \Omega}^2
	+ \frac{\gamma}{2} \| \bx \|_{\bY}^2 \Big).
\end{align}


Introduce an adjoint problem for the state problem \eqref{eq:optimal-abstract} (the choice of loads will be explained later), which is stated as: seek $(\bz, r, \zeta) \in \bbV$ such that
		\begin{subequations} \label{eq:weak-control-adjoint}
			\begin{alignat}{3}
				a_1(\bv, \bz)   +  b_1(\bv, \zeta)     &=& (S^u \by - \bu_d, \bv)_{0, \Omega}  &\quad&\forall \bv \in \bV,\\
				  a_2(r,q)   +    b_2(q,\zeta)  &=&( S^p \by - p_d, q)_{0, \Omega} &\quad&\forall q \in Q,  \\
				b_1(\bz, \phi)    -  b_2(r,\phi)- a_3(\zeta, \phi)  &=& (S^\psi \by- \psi_d, \phi)_{0, \Omega}   &\quad&\forall \phi \in Z.
		\end{alignat}\end{subequations}
    The adjoint solution operator is transpose of the state solution operator then the well-posedness of the adjoint problem \eqref{eq:weak-control-adjoint} can be proved in similar manner as the state problem \eqref{eq:weak-control-state} through the continuity of the functionals $(S^u \by - \bu_d, \bv)_{0, \Omega} , \, ( S^p \by - p_d, q)_{0, \Omega},\, (S^\psi \by- \psi_d, \phi)_{0, \Omega}$ (for more details, refer to \cite{or_sinum16,bociu22}).

\smallskip        
    Our next step is to prove the well-posedness of the continuous problem \eqref{eq:optimal-abstract} and to derive the corresponding optimality condition.
	\begin{theorem}[Existence and uniqueness of an optimal solution] \label{thm:opt-equiv}
    Assume that $\bYd (\subseteq \bY)$ is a non-empty, closed, bounded, and convex set. 
    Then the optimization problem \eqref{eq:optimalproblem} has an optimal solution $\by \in \bYd$.
    Namely, the optimal control problem \eqref{eq:optimal-abstract} has a unique solution
		$(\buu, \by):= (\buS (\by), \by) \in \bbV \times \bYd$ for $\gamma>0$.
		  Furthermore, the first-order optimality condition $j'(\by)(\bx - \by)\ge 0 $ for all $ \bx \in \bYd$ is equivalent to,
		\begin{align}\label{eq:optimality}
			\langle B \bz+ \gamma \by, \bx - \by \rangle_{\bY} \ge 0 \quad \forall \bx \in \bYd,
		\end{align}
		where $\bz \in \bV$ is the adjoint state variable, that is, solution of the problem \eqref{eq:weak-control-adjoint}.
	\end{theorem}
	\begin{proof}
    The existence and uniqueness of solution of the optimal control problem (OCP) follows from \cite[Theorem~$2.14$]{troltzsch10}. This theorem states that OCP \eqref{eq:optimalproblem} with the linear and continuous solution map $\buS$, and closed, bounded and convex set $\bYd \subset \bY$, has a unique solution for $\gamma>0$.
    The problem \eqref{eq:weak-control-state} is linear and hence convex, and Theorem~\ref{thm:Gwelldefined} states that the control to state map is continuous.
    
	\smallskip
    
    For the optimal solution $\by \in \bYd$, we have $j(\by + t (\bx - \by)) \ge j(\by)$ for any $\bx \in \bYd$ and $t \in (0,1]$. Taking $t \to 0$ imply that 
			\begin{align*}
				j'(\by) (\bx - \by) = \lim_{t \to 0} \left(\frac{j(\by+t (\bx - \by)) - j(\by)}{t} \right) \ge 0.
			\end{align*}
			Consider $j(\by+t \bx)$ for some $\bx \in \bYd$, we have
			\begin{align*}
				j(\by+t \bx) & = \frac{1}{2} \left( \| \buu_0 + \buS (\by+t \bx)  - \buu_d \|_{0, \Omega}^2 + \gamma \| (\by+t \bx) \|_{\bY}^2 \right) \\
				& = j(\by) + t \left( (S^u \bx, S^u \by - \bu_d) + (S^p \bx, S^p \by - p_d)  + (S^\psi \bx, S^\psi \by - \psi_d) + \gamma  \langle \by, \bx \rangle_{\bY} \right) \\
				& \qquad + \frac{t^2}{2} \left(  \| S^u \bx \|_{0, \Omega}^2 + \| S^p  \bx  \|_{0, \Omega}^2 +  \| S^\psi \bx \|_{0, \Omega}^2 + \gamma \| \bx \|_{\bY}^2 \right).
			\end{align*}
			Then
			\begin{align} \label{eq:optimalcondn}
				j'(\by)(\bx) = (S^u \bx, S^u \by - \bu_d) + (S^p \bx, S^p \by - p_d)  + (S^\psi \bx, S^\psi \by - \psi_d) + \gamma \langle \by, \bx \rangle_{\bY} .
			\end{align}
			
		In order to further investigate the right-hand side terms in \eqref{eq:optimalcondn}, we consider the adjoint operator defined in \eqref{eq:weak-control-adjoint} corresponding to the problem \eqref{eq:optimal-abstract} with loads, considering the condition \eqref{eq:optimalcondn}. 
		Taking the summation with test functions $\bv= S^u \bx, q = S^p \bx$ and $\phi=S^\psi \bx$
        in \eqref{eq:weak-control-adjoint}, respectively, imply
		\begin{equation}\label{opt-1}
		\begin{aligned}
			&(S^u \bx, S^u \by - \bu_d) + (S^p \bx, S^p \by - p_d)  + (S^\psi \bx, S^\psi \by - \psi_d) \\
			& \,= a_1(S^u \bx, \bz) + b_1(S^u \bx, \zeta) + a_2(r, S^p \bx) +  b_2(S^p \bx, \zeta)+
			b_1(\bz, S^\psi \bx)    -  b_2(r,S^\psi \bx)- a_3(\zeta, S^\psi \bx). 
		\end{aligned}
	\end{equation}
	By definition of solution operator $\buS (\bx)$, we know that
		\begin{subequations} \label{eq:weak-control-S}
			\begin{alignat}{6}
				a_1( S^u \bx, \bv) +  b_1(\bv,  S^\psi \bx)    &=& \langle B \bv, \bx \rangle_{\bY}  &\quad&\forall \bv \in \bV,\\
				a_2(S^p \bx,q) - b_2(q,S^\psi \bx) &=& 0 \qquad&\quad&\forall q \in Q,  \\
				b_1(S^u \bx, \phi)  +  b_2(S^p \bx, \phi) -  a_3(S^\psi \bx, \phi)  &=& 0 \qquad    &\quad&\forall \phi \in Z.
			\end{alignat}
		\end{subequations}
		Summing the above equations with the test functions $\bv= \bz, q = r$ and $\phi=\zeta$ in \eqref{eq:weak-control-S}
        and the use of equation \eqref{opt-1} yield
		\begin{equation}\label{opt-2}
			\begin{aligned}
				&(S^u \bx, S^u \by - \bu_d) + (S^p \bx, S^p \by - p_d)  + (S^\psi \bx, S^\psi \by - \psi_d) 
				 \,=  \langle B \bz, \bx \rangle_{\bY}.
			\end{aligned}
		\end{equation}
		Therefore, the equation \eqref{opt-2} in \eqref{eq:optimalcondn} achieve that, $\forall \bx \in \bYd$,
		\begin{align*}
				0 \le j'(\by)(\bx - \by)  = \langle B \bz + \gamma \by, \bx - \by \rangle_{\bY}.
		\end{align*}
        Thus, the proof is complete.
	\end{proof}

    \bigskip
    As a consequence of above result, we can equivalently write \eqref{eq:optimal-abstract} as the following system:
	seek $(\buu, \by, \buz) \in \bbV \times \bYd\times \bbV $ 
	such that
	\begin{subequations} \label{eq:weak-optimalcontrol}
		\begin{alignat}{6}
			a_1(\bu, \bv)   +  b_1(\bv, \psi) &= & F(\bv) + \langle B \bv, \by \rangle_{\bY}  & \quad & \forall \bv \in \bV, \\
			a_2(p,q)    -    b_2(q,\psi) & = & L(q)  \qquad \quad & \quad & \forall q \in Q, \\
			b_1(\bu, \phi)  + b_2(p,\phi)  - a_3(\psi, \phi) & =& 0  \qquad \qquad& \quad & \forall \phi \in Z,\\
			a_1(\bv, \bz) +  b_1(\bv, \zeta)     &=& (\bu - \bu_d, \bv)_{0, \Omega} \,  &\quad&\forall \bv \in \bV,\\
			a_2(r,q)   +    b_2(q,\zeta)  &=&( p - p_d, q)_{0, \Omega} \, &\quad&\forall q \in Q,  \\
			b_1(\bz, \phi)    -  b_2(r,\phi) - a_3(\zeta, \phi)  &=& (\psi- \psi_d, \phi)_{0, \Omega} \,  &\quad&\forall \phi \in Z, \\
			\langle B \bz + \gamma \by, \bx - \by \rangle_{\bY} &\ge& 0 \qquad \qquad&\quad& \forall \bx \in \bYd.
	\end{alignat}
\end{subequations}

\noindent We introduce the vectorial energy norm $\| \cdot \|_{\bbV}: \bbV \to \bbR$ for any $\buv=(\bv, q, \phi) \in \bbV$ as
\begin{equation*}
    \|\buv\|_{\bbV}^2:= \| \bv\|_{1, \Omega}^2 + \| q\|_{1, \Omega}^2 + \|\phi \|_{0, \Omega}^2.
\end{equation*}

\begin{remark}
    The distributed control or Neumann boundary control problem for a scalar pressure variable (that is $\ell$ as $\ell + y\,  (y \in Y:=L^2(\Omega))$ in \eqref{eq:massconservation}, or $(\kappa/\eta) \nabla p \cdot n^{\partial \Omega} =y\, (y \in Y:=L^2(\Gamma_N))$ in \eqref{bc:Dirichlet}, respectively, for control $y \in Y_d:=\{y \in Y: y \in [y_a,y_b] \}$) can be proceeded in similar manner with appropriate operator $B$ to define the inner product $\langle B q, y\rangle_{Y}$ and test function $q \in Q$.
\end{remark}

In the next step, we formulate the distributed and Neumann boundary control problems by specifying the variational spaces associated with their control variables.


\bigskip
\noindent 
{\bf Distributed control problem:}
We define the control space as $\bY:= [L^2(\Omega)]^2$, and the variational formulation \eqref{eq:weak-optimalcontrol} reduces to:
  seek  $(\buu, \by, \buz) \in \bbV \times \bYd\times \bbV $ 
	 such that
	\begin{subequations} \label{eq:weak-D}
		\begin{alignat}{6}
			a_1(\bu, \bv) +  b_1(\bv, \psi) & = &  (\ff+\by, \bv)_{0, \Omega}  & \quad & \forall \bv \in \bV, \label{eq:weak-D-a}\\
			a_2(p,q)  - b_2(q,\psi) & = & L(q)  \qquad \quad & \quad & \forall q \in Q,  \label{eq:weak-D-b}\\
			b_1(\bu, \phi) + b_2(p,\phi) - a_3(\psi, \phi) & =& 0  \qquad \qquad& \quad & \forall \phi \in Z,  \label{eq:weak-D-c}\\
			a_1(\bv, \bz) +  b_1(\bv, \zeta)     &=& (\bu - \bu_d, \bv)_{0, \Omega} \,  &\quad&\forall \bv \in \bV,  \label{eq:weak-D-d}\\
			a_2(r,q) + b_2(q,\zeta)  &=&( p - p_d, q)_{0, \Omega} \, &\quad&\forall q \in Q,  \label{eq:weak-D-e} \\
			b_1(\bz, \phi) -  b_2(r,\phi) - a_3(\zeta, \phi)  &=& (\psi- \psi_d, \phi)_{0, \Omega} \,  &\quad&\forall \phi \in Z,  \label{eq:weak-D-f}\\
			( \bz + \gamma \by, \bx - \by )_{0,\Omega} &\ge& 0 \qquad \qquad&\quad& \forall \bx \in \bYd. \label{eq:optimal-D}
		\end{alignat}
	\end{subequations}
    Moreover, the optimality condition \eqref{eq:optimal-D} is equivalent to (refer \cite{troltzsch10})
		\begin{align*}
			\by &:= \bPi_{[a,b]} \left( \left({-1}/{\gamma}\right) \bz  \right),  \quad 
            \text{ where }  \bPi_{[a,b]} \boldsymbol{q}(x):= \max \{ \by_a, \min\{ \by_b, \boldsymbol{q}(x) \}\} \quad \forall x \in \Omega. 
		\end{align*}

\bigskip
\noindent 
{\bf Neumann boundary control problem:}
We impose the boundary conditions \eqref{bc:Dirichlet}-\eqref{bc:Neumann} in strong form with $$(2{\mu} \beps(\bu)  -\psi \bI) \nn^{\partial \Omega} =\by \qquad \text{on } \Gamma_N.$$
    Define the variational space for control variable as $\bY:= [L^2(\Gamma)]^2$, and 
    the weak formulation for Neumann control can be stated from \eqref{eq:weak-optimalcontrol} as,
	seek  $(\buu, \by, \buz) \in \bbV \times \bYd\times \bbV $ 
	 such that
	\begin{subequations} \label{eq:weak-N}
		\begin{alignat}{6}
			a_1(\bu, \bv)  +  b_1(\bv, \psi) &= & (\ff,\bv)_{0, \Omega} + (\bv, \by )_{0, \Gamma_N}  & \quad & \forall \bv \in \bV, \label{eq:weak-N-a}\\
			a_2(p,q) - b_2(q,\psi) & = & L(q)  \qquad \quad & \quad & \forall q \in Q, \label{eq:weak-N-b}\\
			b_1(\bu, \phi)  + b_2(p,\phi) - a_3(\psi, \phi) & =& 0  \qquad \qquad& \quad & \forall \phi \in Z, \label{eq:weak-N-c}\\
			a_1(\bv, \bz) +  b_1(\bv, \zeta)     &=&
            (\bu - \bu_d, \bv)_{0, \Omega}
            \,  &\quad&\forall \bv \in \bV, \label{eq:weak-N-d}\\
			a_2(r,q)  +  b_2(q,\zeta)  &=& 
            ( p - p_d, q)_{0, \Omega}
            \, &\quad&\forall q \in Q, \label{eq:weak-N-e} \\
			b_1(\bz, \phi) - b_2(r,\phi) -  a_3(\zeta, \phi)  &=& 
            (\psi- \psi_d, \phi)_{0, \Omega}
            \,  &\quad&\forall \phi \in Z, \label{eq:weak-N-f} \\
			(\bz + \gamma \by, \bx - \by )_{0,\Gamma_N} &\ge& 0 \qquad \qquad&\quad& \forall \bx \in \bYd. \label{eq:optimal-N}
		\end{alignat}
	\end{subequations}
	The optimality condition \eqref{eq:optimal-N} is equivalent to 
	\begin{align*}
		\by & := \bPi_{[a,b]} \left( \left({-1}/{\gamma}\right) \bz  \right), \quad \text{ where } \bPi_{[a,b]} \boldsymbol{q}(x):= \max \{ \by_a, \min\{ \by_b, \boldsymbol{q}(x) \}\} \quad \forall x \in \Gamma_N.
	\end{align*}



	\section{Abstract analysis} \label{sec:abstract}
	
	In what follows, we denote by  ${\mathcal T}_h$ sequences of polygonal decompositions of the domain $\Omega$,
    having mesh-size $h:=\max_{K\in{\mathcal T}_h}h_K$, for any polygon $K$.
	We will follow the analysis further with the optimize-then-discretize strategy. Taking into account the conforming spaces $\bV_h \subset \bV$, $Q_h \subset Q$, $Z_h \subset Z$ and $\bY_h \subset \bY$ ($\bY_{d,h} \subset \bYd$), denoting $\bbV_h:= \bV_h \times Q_h \times Z_h$ and the discrete formulation in terms of $(\buu_h, \by_h, \buz_h):= ((\bu_h, p_h, \psi_h), \by_h, (\bz_h, r_h, \zeta_h)) \in \bbV_h \times \bY_{d,h} \times \bbV_h$ 
	 is expressed as,
	\begin{subequations} \label{eq:weak-h}
		\begin{align}
			a_{1,h}(\bu_h, \bv_h) + b_1(\bv_h, \psi_h) &= F_h(\bv_h) + \langle B_h \bv_h, \by_h \rangle_{\bY}  & \forall \bv_h \in \bV_h, \\
			a_{2,h}(p_h, q_h) - b_2(q_h, \psi_h) & =  \qquad L_h(q_h)  & \forall q_h \in Q_h, \\
			b_1(\bu_h, \phi_h) + b_2(p_h, \phi_h) - a_3(\psi_h, \phi_h) & = \qquad \quad 0  & \forall \phi_h \in Z_h, \\
			a_{1,h}(\bv_h, \bz_h) + b_1(\bv_h, \zeta_h) &= \langle \bu_h - \bu_d, \bv_h \rangle_h  & \forall \bv_h \in \bV_h, \\
			a_{2,h}(r_h, q_h) + b_2(q_h, \zeta_h) & = \langle p_h - p_d, q_h \rangle_h & \forall q_h \in Q_h, \\
			b_1(\bz_h, \phi_h) - b_2(r_h, \phi_h) - a_3(\zeta_h, \phi_h) & = ( \psi_h - \psi_d, \phi_h )_{0, \Omega}  & \forall \phi_h \in Z_h, \\
			\langle B_h \bz_h+ \gamma \by_h, \bx_h - \by_h \rangle_{\bY} & \ge  \qquad \quad 0 &\quad \forall \bx_h \in \bY_{d,h},  \label{eq:optimality-h}
		\end{align}
	\end{subequations}
	where $B_h: \bV+\bV_h \to \bY_{d}$ is a continuous map, the discrete bilinear forms $a_{1,h}( \cdot, \cdot): \bV_h \times \bV_h \to \bbR$, $a_{2,h}( \cdot, \cdot): Q_h \times Q_h \to \bbR$ are bounded and coercive, $b_1(\cdot, \cdot)$ satisfies inf-sup condition on $\bV_h \times Z_h$, bilinear forms $b_2(\cdot, \cdot)$ and $a_3(\cdot, \cdot)$ are bounded, the discrete linear functionals $F_h(\cdot): \bV_h \to \bbR$, $L_h(\cdot): Q_h \to \bbR$ and the discrete linear functional $\langle v , \cdot \rangle_h: \{\bV_h,Q_h\} \to \bbR$ for any $v$ in $L^2(\Omega)$ are analogous to the $L^2$ inner product. Also, we introduce 
    the vectorial norm $\triplenorm{\cdot}_h: \bbV_h \to \bbR$ (analogous to $L^2$ norm) as
    $\triplenorm{\buv_h}_h^2:= \langle \bv_h, \bv_h\rangle_h + \langle q_h, q_h \rangle_h + \langle\phi_h, \phi_h \rangle_h$ for any $\buv_h=(\bv_h, q_h, \phi_h) \in \bbV_h$.

	
	
	
	\subsection{Abstract error analysis}

    We state here the following projections which will be needed in developing the
    the error estimates:
	
	\begin{itemize}
		\item {\bf $\bY$ projection:} Define $\bPi^{Y}: \bY \to \bY_h$ as
		\begin{align} \label{def:Py}
			\langle \bPi^{Y} \by - \by, \bx_h\rangle_{\bY} =0 \qquad \qquad \forall \,\bx_h \in \bY_h.
		\end{align}
		\item {\bf Poroelastic projection:} Introduce the projection operator
		$(P_h, \tP_h)=((P_h^u, P_h^p, P_h^{\psi}), (\tP_h^u,  \tP_h^p,  \tP_h^{\psi})): \bbV \times \bbV \to \bbV_h \times \bbV_h$
		 as
		\begin{subequations}\label{eq:Poro-proj}
			\begin{alignat}{6}
				a_{1,h}(P_h^u \bu, \bv_h) +  b_1(\bv_h, P_h^{\psi} \psi)     &=&\;F_h(\bv_h) + \langle \by, B_h \bv_h \rangle_{\bY}    &\quad&\forall \bv_h \in \bV_h, \label{eq:Poro-proj-a}\\
				a_{2,h}(P_h^p p, q_h) - b_2(q_h, P_h^{\psi} \psi)  &=& L_h(q_h) \qquad \quad   &\quad&\forall q_h \in Q_h, \label{eq:Poro-proj-b} \\
				b_1(P_h^u \bu, \phi_h)  +  b_2(P_h^p p, \phi_h) - a_3(P_h^\psi \psi, \phi_h)  &=&0 \qquad \, \qquad   &\quad&\forall \phi_h \in Z_h, \label{eq:Poro-proj-c} \\
				a_{1,h}(\tP_h^u \bz, \bv_h) + b_1(\bv_h, \tP_h^{\psi} \zeta)     &=&\;\langle \bu - \bu_d, \bv_h \rangle_h  \qquad  &\quad&\forall \bv_h \in \bV_h, \label{eq:Poro-proj-d}\\
				a_{2,h}(\tP_h^p r, q_h) + b_2(q_h, \tP_h^{\psi} \zeta)  &=& \langle p- p_d, q_h \rangle_h  \qquad  &\quad&\forall q_h \in Q_h, \label{eq:Poro-proj-e} \\
				b_1(\tP_h^u \bz, \phi_h) - b_2(\tP_h^p r, \phi_h) - a_3(\tP_h^\psi \zeta, \phi_h)  &=& \,\,( \psi - \psi_d, \phi_h )_{0, \Omega} \qquad  &\quad&\forall \phi_h \in Z_h. \label{eq:Poro-proj-f}
		\end{alignat}\end{subequations}
	\end{itemize}
For given $(\buu, \buz) \in \bbV \times \bbV$, there exists a unique solution $(P_h \buu, \tP_h \buz) \in \bbV_h \times \bbV_h$ satisfying the system of equations \eqref{eq:Poro-proj} (refer to \cite{verma24}).


The next result will be used in what follows.

\begin{lemma} \label{lem:dis-proj}
		Assume that $\buu \in \bbV$ and $P_h \buu \in \bbV_h$ are the solution of continuous problem  \eqref{eq:weak-D-a}-\eqref{eq:weak-D-c} and poroelastic projection \eqref{eq:Poro-proj-a}-\eqref{eq:Poro-proj-c}, respectively.
		 Then it holds, for a positive constant $C$, independent of $h$,
    \begin{align*}
		2 \mu & \| \beps(\bu - P_h^u \bu) \|_{0, \Omega}^2  + c_0 \| p - P_h^p p \|_{0, \Omega}^2 +  \frac{\kappa_1}{\eta} \| \nabla (p - P_h^p p)\|_{0, \Omega}^2+ \| \psi - P_h^{\psi} \psi \|_{0, \Omega}^2 \\
		& \le C \bigg( \inf_{(0 \neq)\buw_h \in \bbV_h}\frac{1}{\|\buw_h\|_{\bbV}} \Big( |(F-F_h)(\bv_h)|  + |(L-L_h)(q_h)| + |\langle \by, (B-B_h) \bv_h \rangle_{\bY}| \Big)  \\
        & \qquad + \left( 1 + \frac{1}{\lambda} + \frac{\alpha}{\lambda} \right) \inf_{\bar{\phi}_h \in Z_h} \|\psi-\bar{\phi}_h\|_{0, \Omega}+  \left(  \frac{\kappa_2}{\eta} + \frac{\alpha}{\lambda} \right) \inf_{\bar{q}_h \in Q_h}  \| p- \bar{q}_h \|_{1, \Omega} + 2 \mu \inf_{\bar{\bv}_h \in \bV_h} \| \bu - \bar{\bv}_h\|_{1, \Omega}  \\
        & \qquad + \sum_{K \in \cT_h} \Big( 2 \mu \|\bu - \bu_{\pi} \|_{1, K}  + \frac{\kappa_2}{\eta} \|p- p_{\pi}\|_{1, K} \Big) \bigg),
    \end{align*}
    where $(\bar{\bv}_h, \bar{q}_h, \bar{\phi}_h) \in \bbV_h$ is an arbitrary function, and local polynomial approximations $\bu_{\pi}$ and $p_{\pi}$ for $\bu$ and $p$, respectively.
\end{lemma}
\begin{proof}
For any function ${\buw}_h:=(\bar{\bv}_h, \bar{q}_h, \bar{\phi}_h) \in \bbV_h$ with equations of poroelastic projection \eqref{eq:Poro-proj}, and use of the continuous problem \eqref{eq:weak-D}, we obtain
	\begin{align*}
		a_{1,h}(P_h^u \bu - \bar{\bv}_h, \bv_h)+  b_1(\bv_h, P_h^{\psi} \psi - \bar{\phi}_h) & = (F_h-F)(\bv_h) + \langle \by, (B_h-B) \bv_h \rangle_{\bY} \\
		& \quad + b_1(\bv_h, \psi-\bar{\phi}_h) + \big( a_1(\bu, \bv_h) - a_{1,h}(\bar{\bv}_h, \bv_h) \big), \qquad \qquad \qquad \qquad \\
		a_{2,h}(P_h^p p - \bar{q}_h, q_h)   -    b_2(q_h, P_h^{\psi} \psi - \bar{\phi}_h) & = (L_h-L)(q_h) + \big( a_2(p, q_h) - a_{2,h}(\bar{q}_h, q_h) \big)  \\
		& \qquad - b_2(q_h, \psi-\bar{\phi}_h), \qquad \qquad \qquad \qquad \qquad \qquad \qquad \\
		b_1(P_h^u \bu - \bar{\bv}_h, \phi_h)    +  b_2(P_h^p p - \bar{q}_h, \phi_h)- a_3(P_h^\psi \psi - \bar{\phi}_h, \phi_h)  & =b_1(\bu - \bar{\bv}_h, \phi_h) +  b_2(p - \bar{q}_h, \phi_h) \\
		& \qquad - a_3(\psi - \bar{\phi}_h, \phi_h).
	\end{align*}
    Taking the test functions as $\bv_h=P_h^u \bu - \bar{\bv}_h$, $q_h=P_h^p p - \bar{q}_h$ and $\phi_h=P_h^{\psi} \psi - \bar{\phi}_h$ alongwith the coercivity of discrete bilinear forms yield,
	\begin{align*}
		& 2 \mu  \| \beps( P_h^u \bu - \bar{\bv}_h) \|_{0, \Omega}^2  + c_0 \| P_h^p p - \bar{q}_h \|_{0, \Omega}^2   + \frac{\kappa_1}{\eta} \| \nabla (P_h^p p - \bar{q}_h)\|_{0, \Omega}^2 \\
        & \lesssim (F_h-F)(\bv_h) + \langle \by, (B_h-B) \bv_h \rangle_{\bY} + b_1(\bv_h, \psi-\bar{\phi}_h) + \big( a_1(\bu, \bv_h) - a_{1,h}(\bar{\bv}_h, \bv_h) \big) + (L_h-L)(q_h) \\
        & \quad + \big( a_2(p, q_h) - a_{2,h}(\bar{q}_h, q_h) \big)  - b_2(q_h, \psi-\bar{\phi}_h) + b_1(\bu - \bar{\bv}_h, \phi_h) +  b_2(p - \bar{q}_h, \phi_h)- a_3(\psi - \bar{\phi}_h, \phi_h).
	\end{align*}
    For any $\buw_h(\neq \underline{\cero}) \in \bbV_h$ and use of inf-sup condition of $b_1(\cdot, \cdot)$ 
    gives
    \begin{align*}
		& 2 \mu  \| \beps( P_h^u \bu - \bar{\bv}_h) \|_{0, \Omega}  + c_0 \| P_h^p p - \bar{q}_h\|_{0, \Omega}^2  + \frac{\kappa_1}{\eta} \| \nabla (P_h^p p - \bar{q}_h)\|_{0, \Omega} + 
        \| P_h^\psi \psi - \bar{\phi}_h \|_{0,\Omega}  \\
        & \lesssim \inf_{( \underline{\cero} \neq)\buw_h \in \bbV_h}\frac{1}{\|\buw_h\|_{\bbV}} \bigg( (F_h-F)(\bv_h) + \langle \by, (B_h-B) \bv_h \rangle_{\bY} + b_1(\bv_h, \psi-\bar{\phi}_h) + \big( a_1(\bu, \bv_h) - a_{1,h}(\bar{\bv}_h, \bv_h) \big)\\
        & \qquad \qquad \qquad \qquad \qquad  + (L_h-L)(q_h)  + \big( a_2(p, q_h) - a_{2,h}(\bar{q}_h, q_h) \big)  - b_2(q_h, \psi-\bar{\phi}_h) + b_1(\bu - \bar{\bv}_h, \phi_h) \\
        & \qquad \qquad \qquad \qquad \qquad +  b_2(p - \bar{q}_h, \phi_h)- a_3(\psi - \bar{\phi}_h, \phi_h) \bigg) \\
        & \lesssim \inf_{( \underline{\cero} \neq)\buw_h \in \bbV_h}\frac{1}{\|\buw_h\|_{\bbV}} \Big( (F_h-F)(\bv_h)  + (L_h-L)(q_h) + \langle \by, (B_h-B) \bv_h \rangle_{\bY} \Big) + \left( 1 + \frac{1}{\lambda} + \frac{\alpha}{\lambda} \right) \|\psi-\bar{\phi}_h\|_{0, \Omega} \\
        & \qquad +  \left(  \frac{\kappa_2}{\eta} + \frac{\alpha}{\lambda} \right) \| p- \bar{q}_h \|_{1, \Omega} + 2 \mu \| \bu - \bar{\bv}_h\|_{1, \Omega} + \sum_{K \in \cT_h} \Big( 2 \mu\|\bu - \bu_{\pi} \|_{1, K}  + \frac{\kappa_2}{\eta} \|p- p_{\pi}\|_{1, K} \Big).
    \end{align*}
    The proof is complete.
\end{proof}

\bigskip

We will follow the above lemma to derive the estimates for the poroelastic projection for the adjoint problem, and add the proof for completion.

\begin{lemma} \label{lem:dis-proj-t}
	Assume that $\buz \in \bbV$ and $\tP_h \buz \in \bbV_h$ are the solution of the continuous problem \eqref{eq:weak-D-a}-\eqref{eq:weak-D-c} and the poroelastic projection \eqref{eq:Poro-proj-d}-\eqref{eq:Poro-proj-f}, respectively. Then it holds, for a positive constant $C$, independent of $h$,
    \begin{align*}
		2 \mu & \| \beps(\bz - \tP_h^u \bz) \|_{0, \Omega}^2  + c_0 \| r - \tP_h^p r \|_{0, \Omega}^2 +  \frac{\kappa_1}{\eta} \| \nabla (r - \tP_h^p r)\|_{0, \Omega}^2+ \| \zeta - \tP_h^{\psi} \zeta \|_{0, \Omega}^2 \\
        & \lesssim  \inf_{(0 \neq)\buw_h \in \bbV_h}\frac{1}{\|\buw_h\|_{\bbV}} \Big( |\langle \bu - \bu_d, \bv_h \rangle_h -  (\bu - \bu_d, \bv_h)_{0, \Omega}|  + |\langle p - p_d, q_h \rangle_h -  (p - p_d, q_h)_{0, \Omega}| \Big) \\
        & \qquad + 2 \mu \inf_{\bar{\bv}_h \in \bV_h} \| \bz - \bar{\bv}_h\|_{1, \Omega} + \left( 1 + \frac{1}{\lambda} + \frac{\alpha}{\lambda} \right) \inf_{\bar{\phi}_h \in Z_h} \|\zeta -\bar{\phi}_h\|_{0, \Omega}+  \left(  \frac{\kappa_2}{\eta} + \frac{\alpha}{\lambda} \right) \inf_{\bar{q}_h \in Q_h}  \| r- \bar{q}_h \|_{1, \Omega}  \\
        & \qquad + \sum_{K \in \cT_h} \Big( 2 \mu\|\bz - \bz_{\pi} \|_{1, K}  + \frac{\kappa_2}{\eta} \|r- r_{\pi}\|_{1, K} \Big),
    \end{align*}
    where $(\bar{\bv}_h, \bar{q}_h, \bar{\phi}_h) \in \bbV_h$ is an arbitrary function, and local polynomial approximations $\bz_{\pi}$ and $r_{\pi}$ for $\bz$ and $r$, respectively.
\end{lemma}

\begin{proof}
     For any function ${\buw}_h:=(\bar{\bv}_h, \bar{q}_h, \bar{\phi}_h) \in \bbV_h$ with use of poroelastic projection \eqref{eq:Poro-proj} and the continuous problem \eqref{eq:weak-D}, we get
    \begin{align*}
		a_{1,h}(\tP_h^u \bz - \bar{\bv}_h, \bv_h) +  b_1(\bv_h, \tP_h^{\psi} \zeta - \bar{\phi}_h) & = \big( \langle \bu - \bu_d, \bv_h \rangle_h -  (\bu - \bu_d, \bv_h)_{0, \Omega} \big) 
        \\
		& \quad + b_1(\bv_h, \zeta-\bar{\phi}_h) + \big( a_1(\bz, \bv_h) - a_{1,h}(\bar{\bv}_h, \bv_h) \big), \qquad \qquad \qquad \qquad \\
		a_{2,h}( \tP_h^p r - \bar{q}_h, q_h) +  b_2(q_h, \tP_h^{\psi} \zeta - \bar{\phi}_h) & = \big( \langle p - p_d, q_h \rangle_h -  (p - p_d, q_h)_{0, \Omega} \big) 
        \\
		& \quad  + \big( a_2(r, q_h) - a_{2,h}(\bar{q}_h, q_h) \big)   + b_2(q_h, \zeta-\bar{\phi}_h), \qquad \qquad \qquad \qquad \qquad \qquad \qquad \\
		b_1( \tP_h^u \bz - \bar{\bv}_h, \phi_h)    -  b_2( \tP_h^p r - \bar{q}_h, \phi_h)- a_3( \tP_h^\psi \zeta - \bar{\phi}_h, \phi_h)  & =b_1(\bz - \bar{\bv}_h, \phi_h) +  b_2(r - \bar{q}_h, \phi_h) \\
		& \quad - a_3(\psi - \bar{\phi}_h, \phi_h).
	\end{align*}
    Taking the test functions as $\bv_h= \tP_h^u \bz - \bar{\bv}_h$, $q_h= \tP_h^p r - \bar{q}_h$ and $\phi_h= \tP_h^{\psi} \zeta - \bar{\phi}_h$ alongwith the coercivity of discrete bilinear forms yields,
	\begin{align*}
		& 2 \mu  \| \beps( \tP_h^u \bz - \bar{\bv}_h) \|_{0, \Omega}^2  + c_0 \| \tP_h^p r - \bar{q}_h \|_{0, \Omega}^2   + \frac{\kappa_1}{\eta} \| \nabla (\tP_h^p r - \bar{q}_h)\|_{0, \Omega}^2 +  \| \tP_h^\psi \zeta - \bar{\phi}_h \|_{0, \Omega}^2 \\
        & \lesssim  \big( \langle \bu - \bu_d, \bv_h \rangle_h -  (\bu - \bu_d, \bv_h)_{0, \Omega} \big) + \big( \langle p - p_d, q_h \rangle_h -  (p - p_d, q_h)_{0, \Omega} \big) 
        + \big( a_1(\bz, \bv_h) - a_{1,h}(\bar{\bv}_h, \bv_h) \big)    \\
        & \quad + b_1(\bv_h, \zeta -\bar{\phi}_h) + \big( a_2(r, q_h) - a_{2,h}(\bar{q}_h, q_h) \big)  + b_2(q_h, \zeta -\bar{\phi}_h) + b_1(\bz - \bar{\bv}_h, \phi_h) -  b_2(r - \bar{q}_h, \phi_h)  \\
        & \quad - a_3(\zeta - \bar{\phi}_h, \phi_h).
	\end{align*}
    Using the polynomial projections and continuity of the discrete bilinear forms, we can conclude the proof.
    
\end{proof}

\bigskip In the following lemma, we obtain the bound for the estimates of the control variable.
\begin{lemma} \label{lem:est-bound-h}
	For any $\bx_h \in \bY_{d,h}$, we have the following inequality:
	\begin{align*}
		\gamma \| \by - \by_h \|_{\bY}^2  \le \langle B_h (\bz_h - \tP_h^u \bz), \by - \by_h \rangle_{\bY} -  \langle B_h \bz_h + \gamma \by_h, \by - \bx_h \rangle_{\bY} - \langle B_h( \bz - \tP_h^u \bz), \by - \by_h \rangle_{\bY}.
	\end{align*}
\end{lemma}
\begin{proof}
	Split the term as $\bz_h - \tP_h^u \bz= (\bz_h - \bz) + (\bz- \tP_h^u \bz)$ then we have
	\begin{align*}
		\langle B_h (\bz_h - \tP_h^u \bz), \by - \by_h \rangle_{\bY}  & = \langle B_h (\bz_h - \bz), \by - \by_h \rangle_{\bY} + \langle B_h(\bz - \tP_h^u \bz), \by - \by_h \rangle_{\bY}.
	\end{align*}
	Using the optimality conditions \eqref{eq:optimality} and \eqref{eq:optimality-h}, the first term on the right hand side can be written, for any $\bx_h \in \bY_{d,h}$, as
	\begin{align*}
		\langle B_h (\bz_h - \bz), \by - \by_h \rangle_{\bY} = & \langle B_h \bz_h + \gamma \by_h, \by - \by_h \rangle_{\bY} - \langle B_h \bz + \gamma \by_h, \by - \by_h \rangle_{\bY} \\
		= & \langle B_h \bz_h + \gamma \by_h, \by - \by_h \rangle_{\bY} - \langle B \bz + \gamma \by, \by - \by_h \rangle_{\bY} + \gamma \|\by - \by_h \|_{\bY}^2 \\
		\ge  & \langle B_h \bz_h + \gamma \by_h, \by - \bx_h \rangle_{\bY}  + \gamma \|\by - \by_h\|_{\bY}^2.
	\end{align*}
\end{proof}

\noindent In next theorem, we derive the error estimates in consequence of the previously stated results.
\begin{theorem} \label{thm:abs-control}
	Assume that $(\buu,\by,\buz) \in \bbV \times \bYd \times \bbV$ and $(\buu_h, \by_h, \buz_h) \in \bbV_h \times \bY_{d,h} \times \bbV_h$ are the solutions of continuous problem \eqref{eq:weak-optimalcontrol} and discrete problem \eqref{eq:weak-h}, respectively. The following estimate holds for a constant $C$, independent of $h$,
	\begin{align*}
		\| \by - \by_h \|_{\bY}^2 & + \triplenorm{ \buu - \buu_h}_h^2 + \| \buz - \buz_h \|_{\bbV}^2\\
		 & \le C \big( \| B \bz - \bPi^Y (B \bz) \|_{\bY}^2 + \| \by - \bPi^Y \by \|_{\bY}^2  + \| \bz - \tP_h^u \bz\|_{0, \Omega}^2 + \triplenorm{ \buu - P_h \buu }_h^2 \big).
	\end{align*}
\end{theorem}
\begin{proof}
	From Lemma~\ref{lem:est-bound-h}, we have
	\begin{align*}
		\gamma \| \by - \by_h \|_{\bY}^2 & \le - \langle B_h (\tP_h^u \bz - \bz_h), \by - \by_h \rangle_{\bY}  - \langle B_h( \bz - \tP_h^u \bz), \by - \by_h \rangle_{\bY} -  \langle B_h \bz_h + \gamma \by_h, \by - \bx_h \rangle_{\bY} \\
        & := -(T_1 +T_2 + T_3).
	\end{align*}
	From the poroelastic projection \eqref{eq:Poro-proj} and discrete formulation \eqref{eq:weak-h}, we get the error equation as
	\begin{equation} \label{eq:er1}
		\begin{aligned}
			a_{1,h}(P_h^u \bu - \bu_h, \bv_h)  +  b_1(\bv_h, P_h^{\psi} \psi - \psi_h)     =& \langle \by - \by_h, B_h \bv_h \rangle_{\bY},\\
			a_{2,h}(P_h^p p - p_h, q_h)   -    b_2(q_h, P_h^{\psi} \psi - \psi_h )  =& 0, \qquad \qquad \\
			b_1(P_h^u \bu - \bu_h, \phi_h) + b_2(P_h^p p - p_h, \phi_h) - a_3(P_h^\psi \psi - \psi_h, \phi_h)  =&0; \qquad  \qquad
		\end{aligned}
	\end{equation}
	And
	\begin{equation} \label{eq:er2}
		\begin{aligned}
			a_{1,h}(\bv_h, \tP_h^u \bz - \bz_h) +  b_1(\bv_h, \tP_h^{\psi} \zeta - \zeta_h)     &= \langle \bu - \bu_h, \bv_h \rangle_h,\\
			a_{2,h}(q_h, \tP_h^p r - r_h)   +    b_2(q_h, \tP_h^{\psi} \zeta - \zeta_h)  &= \langle p - p_h, q_h \rangle_h,  \\
			b_1(\tP_h^u \bz - \bz_h, \phi_h)  -  b_2(\tP_h^p r - r_h, \phi_h) - a_3(\tP_h^\psi \zeta - \zeta_h, \phi_h)  &= 
            \left( \psi- \psi_h, \phi_h \right)_{0, \Omega}.
		\end{aligned}
	\end{equation}
	Taking $\buv_h= \tP_h \buz- \buz_h$ in \eqref{eq:er1} and $\buv_h= P_h \buu- \buu_h$ in \eqref{eq:er2} followed with summing the equations imply
	\begin{align*}
		 T_1 &
        = \langle \bu - \bu_h, P_h^u \bu - \bu_h \rangle_h + \langle p - p_h, P_h^p p - p_h \rangle_h  + \big( \psi- \psi_h, P_h^\psi \psi - \psi_h \big)_{0, \Omega}
        \\
		& = \langle \bu - P_h^u \bu, P_h^u \bu - \bu_h \rangle_h 
        + \langle p - P_h^p p, P_h^p p - p_h \rangle_h + \big( \psi- P_h^\psi \psi, P_h^\psi \psi - \psi_h  \big)_{0, \Omega}
        \\
		& \qquad  
        + \triplenorm{ \buu - P_h \buu}_h^2.
	\end{align*}
Use of Young's inequality gives
\begin{align*} 
T_2 &= \langle B_h( \bz - \tP_h^u \bz), \by - \by_h \rangle_{\bY} \le \frac{1}{2 \gamma}\|  B_h( \bz - \tP_h^u \bz)\|_{\bY}^2 + \frac{\gamma}{2} \| \by - \by_h \|_{\bY}^2.
\end{align*}
For choice of $\bx_h =\bPi^Y \by$, the definition of projection $\bPi^Y$ in \eqref{def:Py} and manipulating the terms lead to
	\begin{align*} 
        T_3 &= \langle B_h \bz_h + \gamma \by_h, \by - \bPi^Y \by \rangle_{\bY} \\
		& = \langle B_h (\bz_h- \bz) + \gamma (\by_h - \by), \by - \bPi^Y \by \rangle_{\bY} + \langle B \bz + \gamma \by, \by - \bPi^Y \by \rangle_{\bY} \\
		& = \langle B_h (\bz_h- \bz) + \gamma (\bPi^Y \by - \by), \by - \bPi^Y \by \rangle_{\bY} + \langle (\bI - \bPi^Y) \big(B \bz + \gamma \by\big), \by - \bPi^Y \by \rangle_{\bY} \\
		& = \langle B_h (\bz_h- \bz), \by - \bPi^Y \by \rangle_{\bY} - \gamma \|\by - \bPi^Y \by \|_{\bY}^2 + \langle (\bI - \bPi^Y) B \bz, \by - \bPi^Y \by \rangle_{\bY} + \gamma \|\by - \bPi^Y \by \|_{\bY}^2\\
		& = - \langle B_h (\tP_h^u \bz - \bz_h), \by - \bPi^Y \by \rangle_{\bY} - \langle B_h (\bz- \tP_h^u \bz), \by - \bPi^Y \by \rangle_{\bY} + \langle (\bI - \bPi^Y) B \bz, \by - \bPi^Y \by \rangle_{\bY}.
	\end{align*}

    Use of the bounds of the terms $T_1- T_3$ gives
    \begin{align*}
        \gamma \| \by - \by_h \|_{\bY}^2 & \le -(T_1 +T_2 + T_3) \\
        & = - \Big( \langle \bu - P_h^u \bu, P_h^u \bu - \bu_h \rangle_h 
        + \langle p - P_h^p p, P_h^p p - p_h \rangle_h  
        +\big( \psi- P_h^\psi \psi, P_h^\psi \psi - \psi_h\big)_{0, \Omega}
        \\
		& \qquad  
        + \triplenorm{ \buu - P_h \buu}_h^2
        + \langle B_h (\tP_h^u \bz - \bz_h), \by - \bPi^Y \by \rangle_{\bY} \\
        & \qquad + \langle B_h (\bz- \tP_h^u \bz), \by - \bPi^Y \by \rangle_{\bY} \Big) \\
        & \quad + \frac{1}{2 \gamma}\|  B_h( \bz - \tP_h^u \bz)\|_{\bY}^2 + \frac{\gamma}{2} \| \by - \by_h \|_{\bY}^2 + \langle (\bI - \bPi^Y) B \bz, \by - \bPi^Y \by \rangle_{\bY}.
    \end{align*}
Thus, the Young's inequality and continuity of the functional $B_h$ give
    \begin{align*}
        & \frac{\gamma}{2} \| \by - \by_h \|_{\bY}^2 + \triplenorm{ \buu - P_h \buu}_h^2
        \\
        & \le - \Big( \langle \bu - P_h^u \bu, P_h^u \bu - \bu_h \rangle_h + \langle p - P_h^p p, P_h^p p - p_h \rangle_h  
        +\big( \psi- P_h^\psi \psi, P_h^\psi \psi - \psi_h\big)_{0, \Omega}\\
        & \qquad + \langle B_h (\tP_h^u \bz - \bz_h), \by - \bPi^Y \by \rangle_{\bY} + \langle B_h (\bz- \tP_h^u \bz), \by - \bPi^Y \by \rangle_{\bY} \Big) \\
        & \quad + \frac{1}{2 \gamma} \|  B_h( \bz - \tP_h^u \bz)\|_{\bY}^2 + \langle (\bI - \bPi^Y) B \bz, \by - \bPi^Y \by \rangle_{\bY} \\
        & \le \frac{1}{2} \Big( \frac{1}{\epsilon_1} \triplenorm{ \buu - P_h \buu}_h^2 + {\epsilon_1} \triplenorm{ P_h \buu - \buu_h }_h^2 + \Big( 2+\frac{1}{\epsilon_2} \Big) \| \by - \bPi^Y \by\|_{\bY}^2 + \frac{\epsilon_2}{2}\|  \tP_h^u \bz - \bz_h\|_{\bY}^2  \\
        & \qquad + \Big( 1+\frac{1}{ \gamma} \Big) \|   \bz - \tP_h^u \bz\|_{\bY}^2 + \|  (\bI - \bPi^Y) B \bz\|_{\bY}^2 \Big).
    \end{align*}
	Again using \eqref{eq:er2} with $\buv_h = \tP_h \buz - \buz_h$ gives
	\begin{align*}
		\triplenorm{  \tP_h \buz - \buz_h }_h  \le C(\mu, \kappa, \eta) \big( \triplenorm{ \buu - P_h \buu_h}_h +  \triplenorm{ P_h \buu - \buu_h}_h \big),
	\end{align*}
    then we have
    \begin{align*}
         & \frac{\gamma}{2} \| \by - \by_h \|_{\bY}^2 + \left( 1 - \frac{\epsilon_1}{2}- C(\mu, \kappa, \eta) \frac{\epsilon_2}{2} \right) \triplenorm{ P_h \buu - \buu_h }_h^2 \\
        & \le \frac{1}{2} \bigg( \left(\frac{1}{\epsilon_1} + C(\mu, \kappa, \eta)\frac{\epsilon_2}{2}\right) \triplenorm{ \buu - P_h \buu}_h^2  + \Big( 2+\frac{1}{\epsilon_2} \Big) \| \by - \bPi^Y \by\|_{\bY}^2 + \Big( 1+\frac{1}{ \gamma} \Big) \|  \bz - \tP_h^u \bz\|_{\bY}^2 \\
        & \qquad + \|  (\bI - \bPi^Y) B \bz\|_{\bY}^2 \bigg).
    \end{align*}
	Choose $\epsilon_1 =\frac{1}{2}$ and $\epsilon_2 = \frac{1}{2C(\mu, \kappa, \eta)}$ then we conclude that
	\begin{align*}
		\gamma \| \by - \by_h \|_{\bY}^2 +  \triplenorm{ P_h \buu - \buu_h}_h^2 \le \Big( & \| (\bI - \bPi^Y) B \bz \|_{\bY}^2 + 2(1+C(\mu, \kappa, \eta))\| \by - \bPi^Y \by \|_{\bY}^2  \\
        & + \Big( 1+\frac{1}{ \gamma} \Big) \| \bz - \tP_h^u \bz\|_{0, \Omega}^2 + \frac{9}{4} \triplenorm{ \buu - P_h \buu }_h^2 \Big).
	\end{align*}
\end{proof}

In the following corollary, we state the error estimates of the state variables in the energy norm.
\begin{corollary} \label{cor:abs-state}
	Assume that $(\buu,\by,\buz) \in \bbV \times \bYd \times \bbV$ and $(\buu_h, \by_h, \buz_h) \in \bbV_h \times \bY_{d,h} \times \bbV_h$ are the solutions of equations \eqref{eq:weak-optimalcontrol} and \eqref{eq:weak-h}, respectively. The following estimate holds,
	\begin{align*}
		\| \buu - \buu_h \|_{\bbV}^2 \le C \Big(  \| B \bz - \bPi^Y (B \bz) \|_{\bY}^2 + \| \by - \bPi^Y \by \|_{\bY}^2 
		+ \| \bz - \tP_h^u \bz \|_{0, \Omega}^2 + \| \buu - P_h \buu \|_{\bbV}^2 \Big).
	\end{align*}
\end{corollary}

\begin{proof}
	Split the error as $\| \buu - \buu_h \|_{\bbV} \le \| \buu - P_h \buu_h\|_{\bbV} +  \| P_h \buu - \buu_h\|_{\bbV}.$
	Taking $\buv_h= P_h \buu- \buu_h$ in \eqref{eq:er1} and use of Cauchy-Schwarz, continuity of $B_h$ and Young's inequality leads to
	\begin{align*}
			\| P_h \buu- \buu_h \|_{\bbV} \lesssim \|\by - \by_h \|_{\bY}.
	\end{align*}
	Thus, the required estimates follow from Theorem~\ref{thm:abs-control}.
\end{proof}

	\section{Virtual element approximation} \label{sec:VEapprox}

    In this section, we will write a VEM discretization for the control problems. With this
aim, we start with the mesh construction and the assumptions considered to introduce
the discrete virtual element spaces.
	
	%
	We denote by  ${\mathcal T}_h$ sequences of polygonal decompositions of the domain $\Omega$, having mesh-size $h:=\max_{K\in{\mathcal T}_h}h_K$.
	By $N^v_K$ we will denote the number of vertices in the polygon $K$, $N^e_K$
	will stand for the number of edges on $\partial K$, and $e$ a generic
	edge of $\mathcal{T}_h$ and set of all edges $e$ by $\mathcal{E}_h$. For all $e\in \partial K$, we denote by $\boldsymbol{n}_K^e$
	the unit normal pointing outwards $K$, $\boldsymbol{t}^e_K$ the unit tangent vector along $e$
	on $K$, and $V_i$ represents the $i^{th}$ vertex of the polygon $K$. By $\mathcal{E}_h^D$ and $\mathcal{E}_h^N$, we denote the set of all edges $e$ in the mesh $\cT_h$ on the Dirichlet boundary $\Gamma_D$ and the Neumann boundary $\Gamma_N$. 
	
	\subsection{Virtual element spaces and degrees of freedom} \label{subsec:VEspaces}
	
	As in \cite{beirao-div17,verma24} we  suppose regularity of the polygonal meshes in the following sense:
	there exists $C_{{\mathcal T}}>0$ such that, for every $h$ and every $K\in {\mathcal T}_h$,
	the following holds
	\begin{itemize}
		\item[($A$)] $K\in{\mathcal T}_h$ is star-shaped with respect to every point within a ball of radius $C_{{\mathcal T}}h_K$;
		\item[($\tilde{A}$)] the ratio between the shortest edge and $h_K$ is larger than $C_{{\mathcal T}}$.
	\end{itemize}

	Denoting by $\mathbb{P}_k(K)$ the space of polynomials of degree up to $k$, defined locally on $K\in\cT_h$,
	we proceed to characterise the scalar energy projection operator $\Pi_{K}^{\nabla}: H^1(K) \rightarrow \mathbb{P}_1(K)$ by the relations
	\begin{equation*}
		(\nabla (\Pi_{K}^{\nabla} q - q), \nabla r)_{0,K} = 0, \qquad P^0_K(\Pi_{K}^{\nabla} q - q)=0,
	\end{equation*}
	valid for all $q \in H^1(K)$ and $r \in \mathbb{P}_1(K)$, and where
	$ P^0_K(q):= \int_{\partial K} q\;\dx$.
	
	Denoting by $\mathcal{M}_k(K)$ the space of monomials of degree up to $k$  defined locally on  $K \in \cT_h$, we can define the local virtual element spaces for global displacement, fluid pressure, and global pressure for $k=1$ on each polygon $K \in \cT_h$
	  as follows
	\begin{align} \label{VE-spaces}
		\begin{split}
			\bV_h(K) & :=   \Biggr\{ \bv_h \in \mathbb{B}(\partial K):
				\begin{cases}
					- \bDelta \bv_h + \nabla s =0 ~ \text{ for some } s \in L^2(K) \\
					\vdiv \, \bv_h \in \mathbb{P}_0(K), \\
                    (\bPi_{K}^{\nabla,k} \bv_h - \bv_h, \boldsymbol{m}_{\alpha})_{0,K} = 0\ \forall \boldsymbol{m}_{\alpha} \in [\mathcal{M}_1 (K)]^2 \backslash \{ \cero \}
				\end{cases}\Biggr\}, \\
			Q_h(K) & := \Biggl\{ q_h \in H^1(K) \cap C^0(\partial K): 
            \begin{cases}
            \Delta q_h \in \mathbb{P}_1(K), \,\, q_h|_e \in \mathbb{P}_1(e) \,\forall e \in \partial K \\
(\Pi_{K}^{\nabla,k} q_h - q_h, m_{\alpha})_{0,K} = 0\ \forall m_{\alpha} \in \mathcal{M}_1 (K) \backslash \{ 0 \}             
            \end{cases}
			\Biggr\},\\
			Z_h(K) & := \mathbb{P}_0(K),
	\end{split} 	\end{align}
where $\mathbb{B}(\partial K):= \{ \bv_h \in [H^1(K) \cap C^0(\partial K)]^2: \bv_h \cdot \nn^e_K|_e \in \bbP_2(e), \, \bv_h \cdot \boldsymbol{t}^e_K|_e \in \bbP_1(e) \, \forall e \in \partial K\}.$
	The dimension of $\bV_h(K)$ is $3 N^v_K$, the dimension of $Q_h(K)$ is $N^v_{K}$, and that of $Z_h(K)$ is $1$. 
	
	Next we specify the degrees of freedom associated with \eqref{VE-spaces}. That is, discrete functionals of the type (taking as an example the space for global pressure)
	$$(D_i): Z_{h|K} \to \mathbb{R}; \qquad  Z_{h|K} \ni \phi \mapsto D_i(\phi),$$
	and we start with the degrees of freedom for the local displacement space $\bV_h(K)$, consisting of
	($D_v1$) the values of a discrete displacement $\bv_h$ at vertices of the element;
	and ($D_v2$) the moments of $\bv_h$ on the edges: $\frac{1}{|e|} \int_e \bv_h \cdot \nn \ds.$
	Then we precise the degrees of freedom for the local fluid pressure space $Q_h(K)$:
	($D_q$) the values of $q_h$ at vertices of the polygonal element.
	And similarly, the degrees of freedom for the local global-pressure space $Z_h(K)$. That is, 
	($D_z$) the moment:
	$\frac{1}{|K|} \int_K \phi_h.$
	
	It has been proven elsewhere (e.g.,  \cite{ahmad13} ) that these degrees of
	freedom are unisolvent in their respective spaces. We also define global counterparts of the local VE spaces, as follows
	\begin{gather*}
		\bV_h : = \lbrace \bv_h \in \bV : \bv_h |_K \in \bV_h(K) \; \forall K \in \mathcal{T}_h \rbrace, \qquad Q_h := \lbrace q_h \in Q : q_h |_K \in Q_h(K) \; \forall K \in \mathcal{T}_h \rbrace,\\
	\text{and }	\qquad Z_h := \lbrace \phi_h \in Z : \phi_h |_K \in Z_h(K) \; \forall K \in \mathcal{T}_h \rbrace.
	\end{gather*}
	In addition, $N^\bV$ denotes the number of degrees of freedom for $\bV_h$, $N^Q$ the number of degrees of freedom for $Q_h$,
	$N^Z$ the number of degrees of freedom for $Z_h$,
	and $\text{dof}_r(s)$ stands for the $r$-th degree of a given function $s$.

    Now, we can define the computable projections for the discrete bilinear forms. The energy projection $\bPi^{\beps}_K : \bV_h(K) \rightarrow [\mathbb{P}_1(K)]^2$ such that
	\begin{align*}
	& a_1^K(\bPi^{\beps}_K \bv - \bv , \boldsymbol{r}) = 0 \quad 
	\text{for all $\bv \in \bV_h(K)$ and $\boldsymbol{r} \in [\mathbb{P}_1(K)]^2$,}\\
	& m^K(\bPi^{\beps}_K \bv - \bv,\boldsymbol{r}) = 0 \quad
	\text{for all $\boldsymbol{r} \in  \ker (a_1^K(\cdot,\cdot))$,}
	\end{align*}	
	where
	\begin{align*}
	m^K (\bv,\boldsymbol{r}):= \frac{1}{N^v_K} \sum_{i=1}^{N^v_K} \bv(V_i) \cdot \boldsymbol{r}(V_i).
	\end{align*}
     And the $L^2$-projection on the scalar space as $\Pi^0_K: L^2(K) \rightarrow \mathbb{P}_1(K)$ such that
	\begin{align*}
	(\Pi_{K}^0 q - q, r)_{0,K} = 0, \quad q \in L^2(K), r \in \mathbb{P}_1(K),
	\end{align*}
	and we can clearly verify that $\Pi_{K}^0 q_h =   \Pi_{K}^{\nabla} q_h, \; \forall q_h \in Q_h$. In similar manner, we can define the vectorial $L^2$ projection $\bPi^0_K: [L^2(K)]^2 \rightarrow [\mathbb{P}_1(K)]^2$ onto piecewise linear, and projection $\bPi^{0,0}_K: [L^2(K)]^2 \rightarrow [\mathbb{P}_0(K)]^2$ onto piecewise constants. 
    
    Define the discrete local bilinear forms for $K \in \cT_h$ and for all $\bu_h, \bv_h \in \bV_h,\, p_h, q_h \in Q_h$ and $\psi_h, \phi_h \in Z_h$ as
	\begin{gather*}
		a_{1,h}(\bu_h, \bv_h)|_K := 2 \mu \Big( ( \beps( \bPi^{\beps}_K \bu_h), \beps( \bPi^{\beps}_K \bv_h))_{0,K} + S^{\beps}_K ( (\bI - \bPi^{\beps}_K) \bu_h, (\bI - \bPi^{\beps}_K) \bv_h) \Big),  \\
		\qquad b_1(\bv_h, \phi_h) |_K := - (\vdiv \bv_h,  \phi_h)_{0, K}, \qquad
		F_h(\bv_h)|_K := (\ff ,  \bPi^{0}_K \bv_h )_{0, K}, \\
		m_h(p_h,q_h)|_K  := \biggl( c_0 +\frac{\alpha^2}{{\lambda}}\biggr) \Big(
		( \Pi^{0}_K p_h, \Pi^{0}_K  q_h)_{0, K} + S^{0}_K( (\boldsymbol{I} - \Pi^{0}_K) p_h, (\boldsymbol{I} - \Pi^{0}_K) q_h)  \Big), \\
		\tilde{a}_h(p_h,q_h) |_K  :=\frac{\bar{\kappa}}{\eta}  \Big((\nabla \Pi^{\nabla}_K p_h, \nabla \Pi^{\nabla}_K q_h)_{0, K} + S^{\nabla}_K( (\boldsymbol{I} - \Pi^{\nabla}_K) p_h, (\boldsymbol{I} - \Pi^{\nabla}_K) q_h)  \Big),
		\\
		a_{2,h}(p_h,q_h)|_K:= m_h(p_h,q_h)|_K +\tilde{a}_h(p_h,q_h) |_K , \qquad L_h(q_h) :=  \sum_{K \in \cT_h} ( \ell, \Pi^{0}_K q_h)_{0, K},\\
		b_2(p_h,\phi_h)|_K:=  \frac{\alpha}{{\lambda}}  ( \Pi^{0}_K p_h, \phi_h)_{0, K}, \qquad \text{and} \qquad
		a_3(\psi_h,\phi_h) |_K:=  \frac{1}{\lambda} (\psi_h, \phi_h)_{0,K},
	\end{gather*}
	where the stabilization terms $S^{\beps}_K(\cdot, \cdot), S^{\nabla}_K(\cdot, \cdot), S^0_K(\cdot, \cdot)$ are classical dofi-dofi stabilization acting on the kernel of their respective operators $\bPi^{\beps}_K,\; \Pi^{\nabla}_K,\; \Pi^0_K$. More precisely, we will consider the following:
    \begin{align*}
	S^{\beps}_K(\bu_h,\bv_h) &:= \sum_{l=1}^{N^V} \text{dof}_l(\bu_h) \text{dof}_l(\bv_h), \quad \bu_h, \bv_h \in \text{ker}(\bPi^{\beps}_K),\\
	S^{\nabla}_K(p_h,q_h) &:= \sum_{l=1}^{N^Q} \text{dof}_l(p_h) \text{dof}_l(q_h), \quad p_h, q_h \in \text{ker} (\Pi^{\nabla}_K), \\
	S^0_K(p_h,q_h) &:= \text{area}(K) \sum_{l=1}^{N^Q} \text{dof}_l(p_h) \text{dof}_l(q_h), \quad p_h, q_h \in \text{ker} (\Pi^0_K),
	\end{align*}
    holding the stability, with constants independent of $K$,
	\begin{gather*}
		 C^{\beps}_1 \| \beps(\bv_h) \|_{0, K}^2 \le S^{\beps}_K ( \bv_h, \bv_h) \le  C^{\beps}_2 \| \beps(\bv_h) \|_{0, K}^2, \\
         C^{\nabla}_1\| \nabla q_h \|_{0, K}^2  \le S^{\nabla}_K( q_h, q_h)  \le  C^{\nabla}_2 \| \nabla q_h \|_{0, K}^2,\\
		 C^0_1 \| q_h \|_{0, K}^2  \le  S^{0}_K( q_h, q_h)  \le  C^0_2\| q_h \|_{0, K}^2.
	\end{gather*}
    In addition, the load functionals $\langle f, \cdot \rangle_h$, for any $L^2$ integrable scalar and vector function $f$, are defined, respectively, as
    $$\langle \cdot, q_h \rangle_h:= \sum_{K \in \cT_h} (\cdot, \Pi^0_K q_h )_{0,K}, \quad \text{and} \quad \langle \cdot, \bv_h \rangle_h:= \sum_{K \in \cT_h} ( \cdot, \bPi^0_K \bv_h )_{0,K}.$$
	
	\noindent {\bf Properties of the discrete bilinear forms:}
	\begin{itemize}
		\item $a_{1,h}(\cdot, \cdot)$ is coercive on $\bV_h$: for all $\bv_h \in \bV_h$,
		\begin{align*}
			a_{1,h}(\bv_h, \bv_h) & \ge \sum_{K \in \cT_h} 2 \mu \Big( \| \beps( \bPi^{\beps}_K \bv_h)\|^2_{0,K} + C^{\epsilon}_1\|(\bI - \bPi^{\beps}_K) \bv_h\|_{0,K}^2 \Big) \\
			& \ge 2 \mu \min \{ 1, C^{\epsilon}_1\} \| \beps(\bv_h) \|_{0, \Omega}^2.
		\end{align*}
		
		\item $a_{2,h}(\cdot, \cdot)$ is coercive on $Q_h$: for all $q_h \in Q_h$,
		\begin{align*}
			a_{2,h}(q_h,q_h) &:= \sum_{K \in \cT_h} \bigg( \biggl( c_0 +\frac{\alpha^2}{{\lambda}}\biggr) \Big(
			\|\Pi^{0}_K  q_h\|^2_{0, K} + S^{0}_K( (\boldsymbol{I} - \Pi^{0}_K) p_h, (\boldsymbol{I} - \Pi^{0}_K) q_h)  \Big) \\
			& \qquad \qquad +\frac{\bar{\kappa}}{\eta}  \Big(\| \nabla \Pi^{\nabla}_K q_h\|^2_{0, K} + S^{\nabla}_K( (\boldsymbol{I} - \Pi^{\nabla}_K) p_h, (\boldsymbol{I} - \Pi^{\nabla}_K) q_h)  \Big) \bigg) \\
			& \ge \min \{ 1, C^0_1, C^{\nabla}_1\} \left( \biggl( c_0 +\frac{\alpha^2}{{\lambda}}\biggr) \| q_h\|^2_{0, \Omega} +\frac{\kappa_1}{\eta} \| \nabla q_h\|^2_{0, \Omega} \right).
		\end{align*}
		\item $b_1(\cdot, \cdot)$ satisfies inf-sup condition on $\bV_h \times Q_h$: there exists a $\beta>0$ such that for any $\phi_h \in Z_h$, we have
		\begin{align*}
			\sup_{(0 \neq)\bv_h \in \bV_h} \frac{b_1(\bv_h, \phi_h)}{\| \beps(\bv_h)\|_{0, \Omega}} \ge \beta \| \phi_h \|_{0, \Omega}.
		\end{align*}
	\end{itemize}
	
	\subsection{Discrete distributed formulation}
	Define the discrete space for control as
	\begin{align*}
		\bY_h &:= \lbrace \by_h \in \bY (=[L^2(\Omega)]^2) : \by_h |_K \in [\bbP_0(K)]^2 \, \forall K \in \mathcal{T}_h \rbrace, \\
        \bY_{d,h} & := \lbrace \by_h \in \bY_h: \by_a \le \by_h \le \by_b \rbrace.
	\end{align*}
    Then the virtual element formulation of the continuous distributed control problem \eqref{eq:weak-D} is defined as: seek $(\buu_h, \by_h, \buz_h) \in \bbV_h \times \bY_{d,h} \times \bbV_h$	such that
	\begin{subequations} \label{eq:weak-h-D}
		\begin{align}
			a_{1,h}(\bu_h, \bv_h) + b_1(\bv_h, \psi_h) &=  \sum_{K \in \cT_h} ( \ff + \by_h, \bPi^{0}_K \bv_h )_{0, K}    & \forall \bv_h \in \bV_h, \label{eq:weak-h-D-a}\\
			a_{2,h}(p_h, q_h) - b_2(q_h, \psi_h) & = \sum_{K \in \cT_h} ( \ell, \Pi^{0}_K q_h )_{0, K}  & \forall q_h \in Q_h, \label{eq:weak-h-D-b}\\
			b_1(\bu_h, \phi_h) + b_2(p_h, \phi_h) - a_3(\psi_h, \phi_h) & = \qquad 0  & \forall \phi_h \in Z_h, \label{eq:weak-h-D-c}\\
			a_{1,h}(\bv_h, \bz_h) + b_1(\bv_h, \zeta_h) &= \sum_{K \in \cT_h} ( \bu_h - \bu_d, \bPi^{0}_K \bv_h)_{0, K}   & \forall \bv_h \in \bV_h, \label{eq:weak-h-D-d} \\
			a_{2,h}(r_h, q_h) + b_2(q_h, \zeta_h) & = \sum_{K \in \cT_h} (p_h - p_d, \Pi^{0}_K  q_h)_{0, K}  & \forall q_h \in Q_h, \label{eq:weak-h-D-e} \\
			b_1(\bz_h, \phi_h) - b_2(r_h, \phi_h) - a_3(\zeta_h, \phi_h) & =\sum_{K \in \cT_h} (\psi_h - \psi_d,  \phi_h)_{0, K} \,\, & \forall \phi_h \in Z_h, \label{eq:weak-h-D-f} \\
			\sum_{K \in \cT_h} \left( \bz_h + \gamma \by_h, \bx_h - \by_h \right)_{0,K} & \ge  0 &\quad \forall \bx_h \in \bY_{d,h}.  \label{eq:optimality-h-D}
		\end{align}
	\end{subequations}
	
	
	\noindent The optimality condition \eqref{eq:optimality-h-D} is equivalent to $\by_h|_K:= \bPi_{[a,b]} \left((-1/\gamma) \bPi^{0,0}_K \bz_h|_K  \right)$ for each $K \in \cT_h$.
	
	\subsubsection{A priori error estimates}
	We assume the following regularity of given data,
	\begin{align*}
		\ff \in [L^2(\Omega)]^2, \quad  \ell \in L^2(\Omega), \quad  \bu_d \in [L^2(\Omega)]^2,  \quad p_d, \, \psi_d \in L^2(\Omega), \quad \text{and} \quad  \by \in [H^s(\Omega)]^2. 
	\end{align*}
	Moreover, we assume that the solution of the continuous problem \eqref{eq:strong-poro} satisfies the following regularity estimates with $s \in (0,1]$:
	\begin{align} \label{est:reg}
		(2 \mu)^{1/2} |\bu|_{1+s, \Omega} + \kappa_1^{1/2} |p|_{1+s, \Omega} + \left( c_0^{1/2} + \frac{\alpha}{\lambda^{1/2}}\right) |p|_{s, \Omega} + |\psi|_{s, \Omega} \le C_\text{reg} \big( \| \ff \|_{0, \Omega}
        + \| \ell \|_{0, \Omega} \big).
	\end{align}
    Also, we have the following estimate for the dual/costate variables:
    \begin{align} \label{est:dual-reg}
		(2 \mu)^{1/2} |\bz|_{1+s, \Omega} + \kappa_1^{1/2} |r|_{1+s, \Omega} + \left( c_0^{1/2} + \frac{\alpha}{\lambda^{1/2}}\right) |r|_{s, \Omega} + |\zeta|_{s, \Omega} \qquad \qquad \qquad \qquad   \nonumber \\
        \le \bar{C}_\text{reg} \big( \| \bu_d \|_{0, \Omega}
        + \| p_d \|_{0, \Omega} + \| \psi_d \|_{0, \Omega} \big).
	\end{align}
	
	\noindent We first recall the following classical results (see \cite{brenner08,CGPS}):
	\begin{lemma}[Polynomial Estimates] \label{lem:polynomial}
		For $v \in H^{1+s}(\Omega), \, s \in (0,1]$ and for each $K \in \cT_h$, there exists a polynomial approximation $v_{\pi}|_K \in \bbP_k(K)$ for $v$ and satisfies
		\begin{align*}
			\sum_{K \in \cT_h} \left( \| v - v_{\pi} \|_{0,K} + h_K | v - v_{\pi} |_{1,K} \right) \lesssim h^{1+s} |v|_{1+s, \Omega}.
		\end{align*}
	\end{lemma}
	
	\begin{lemma}[Interpolant Estimates] \label{lem:interpolant}
		 For $v \in H^{1+s}(\Omega), \, s \in (0,1]$, there exists an interpolant $v_{I} \in V_h$ for $v \in V$ and satisfies
		\begin{align*}
			\| v - v_{I} \|_{0, \Omega} + h | v - v_{I} |_{1, \Omega} \lesssim h^{1+s} |v|_{1+s, \Omega}.
		\end{align*}
	\end{lemma}

    \begin{lemma}
        Under the assumption \eqref{est:reg}, we have the estimate of the projection $\bPi^Y$ for a constant $C$, independent of $h$, as
        \begin{align*}
            \| \bz - \bPi^Y \bz \|_{0, \Omega} \le C h^{1 + s} | \bz |_{1+s, \Omega}, \qquad \text{and} \qquad 
            \| \by - \bPi^Y \by \|_{0, \Omega} \le C h^{ s} | \by |_{s, \Omega}.
        \end{align*}
    \end{lemma}
	Now, we define a new poroelastic projection operator as follows,
	\begin{itemize}
	\item {\bf Poroelastic projection:} Define $P_h:=(P_h^u, P_h^p, P_h^{\psi}): \bbV \to \bbV_h$ as
	\begin{subequations}\label{eq:Poro-proj-D}
		\begin{alignat}{6}
			a_{1,h}(P_h^u \bu, \bv_h)+  b_1(\bv_h, P_h^{\psi} \psi)     &=& \sum_{K \in \cT_h} (\ff+ \by, \bPi^{0}_K \bv_h )_{0,K}   &\quad&\forall \bv_h \in \bV_h,\\
			a_{2,h}(P_h^p p, q_h)   -    b_2(q_h, P_h^{\psi} \psi)  &=& \sum_{K \in \cT_h} (\ell, \Pi^{0}_K q_h )_{0,K}  \quad &\quad&\forall q_h \in Q_h,  \\
			b_1(P_h^u \bu, \phi_h)  +  b_2(P_h^p p, \phi_h) - a_3(P_h^\psi \psi, \phi_h)  &=&0 \qquad  \qquad   \quad &\quad&\forall \phi_h \in Z_h.
	\end{alignat}\end{subequations}
	
	\item {\bf Poroelastic projection for adjoint problem:} Define $\tP_h:=(\tP_h^u, \tP_h^p, \tP_h^{\psi}): \bbV \to \bbV_h$ as
	\begin{subequations}\label{eq:Poro-adjoint-D}
		\begin{alignat}{6}
			a_{1,h}(\bv_h, \tP_h^u \bz) +  b_1(\bv_h, \tP_h^{\psi} \zeta)     &=&\; \sum_{K \in \cT_h} ( \bu - \bu_d, \bPi^{0}_K \bv_h )_{0,K}  &\quad&\forall \bv_h \in \bV_h,\\
			a_{2,h}(q_h, \tP_h^p r)   +    b_2(q_h, \tP_h^{\psi} \zeta)  &=& \sum_{K \in \cT_h} ( p- p_d, \Pi^{0}_K q_h )_{0,K}  &\quad&\forall q_h \in Q_h,  \\
			b_1(\tP_h^u \bz, \phi_h)    -  b_2(\tP_h^p r, \phi_h)- a_3(\tP_h^\psi \zeta, \phi_h)  &=& \sum_{K \in \cT_h} ( \psi- \psi_d, \phi_h )_{0,K} \quad    &\quad&\forall \phi_h \in Z_h.
	\end{alignat}\end{subequations}
	\end{itemize}
	The right hand sides of the systems \eqref{eq:Poro-proj-D} and \eqref{eq:Poro-adjoint-D} are continuous linear functionals. Hence, we obtain the unique solutions of both systems respectively (refer \cite{burger_acom21} for more details).
	
	 Next results on the estimates for the poroelastic projections follows from Lemma~\ref{lem:dis-proj} and Lemma~\ref{lem:dis-proj-t} by estimates of $L^2$ projection and interpolant estimates.
	\begin{corollary} \label{lem:proj}
		Assume that $\buu \in \bbV$ and $P_h \buu \in \bbV_h$ are the solution of continuous problem  \eqref{eq:weak-D-a}-\eqref{eq:weak-D-c} and poroelastic projection \eqref{eq:Poro-proj-D} respectively.
		 Then it holds, for constant $C$ independent of $h$,
		\begin{equation}\label{est:proj-D}
			\begin{aligned} 
			2 \mu \| \beps(\bu - P_h^u \bu) \|_{0, \Omega}^2  + c_0 \| p - P_h^p p \|_{0, \Omega}^2 +& \frac{\kappa_1}{\eta} \| \nabla (p - P_h^p p)\|_{0, \Omega}^2+ \| \psi - P_h^{\psi} \psi \|_{0, \Omega}^2 \\
			& \le C h^{2s} ( \|\by \|_{0, \Omega}^2
			+ \|\ff \|_{0, \Omega}^2 + \|\ell\|_{0, \Omega})^2).
		\end{aligned}
	\end{equation}
	for $s \in (0,1]$.
	\end{corollary}

	\begin{corollary} \label{lem:proj-adjoint}
	Assume that $\buz \in \bbV$ and $\tP_h \buz \in \bbV_h$
	 satisfies the equations \eqref{eq:weak-D-d}-\eqref{eq:weak-D-f} and \eqref{eq:Poro-adjoint-D} respectively, then it holds for constant $C$, independent of $h$, and for $s \in (0, 1]$ as
	 \begin{equation}\label{est:adjointproj-D}
	 	\begin{aligned} 
	 		2 \mu \| \bz - \tP_h^u \bz \|_{0, \Omega}^2 + c_0 \| r - \tP_h^p r \|_{0, \Omega}^2 + &\frac{\kappa_1}{\eta} \| \nabla (r - \tP_h^p r )\|_{0, \Omega}^2+ \| \zeta - \tP_h^{\psi} \zeta \|_{0, \Omega}^2 \\
	 		&\qquad \le C h^{2s} \left( \|\by \|_{0, \Omega}^2 + \|\ff \|_{0, \Omega}^2 + \|\ell\|_{0, \Omega}^2 + \| \buu_d\|_{0, \Omega}^2 \right).
	 	\end{aligned}
	 \end{equation}
	\end{corollary}
\noindent Next, we establish the error estimates for the control variable following the abstract analysis for distributed control problem.
\begin{theorem} \label{thm:control-est}
	Assume $(\buu, \by, \buz) \in \bbV \times \bYd \times \bbV$ is solution of problem 
    \eqref{eq:weak-D} and $(\buu_h, \by_h, \buz_h) \in \bbV_h \times \bY_{d,h} \times \bbV_h$ is solution of problem \eqref{eq:weak-h-D} then, for positive constant $C$, independent of $h$, and for $s \in (0, 1]$
	\begin{equation}\label{est:control}
	\begin{aligned} 
		\| \by - \by_h\|_{0, \Omega}^2 +  \sum_{K \in \cT_h} \big( \| \bPi^{0}_K (\bu - \bu_h) \|_{0, K}^2 & + \| \Pi^{0}_K (p-p_h) \|_{0, K}^2  + \| \psi - \psi_h \|_{0, K}^2  \big)\\
		& \quad \le C h^{2s} \left( |\by|_{s, \Omega}+ \|\ff\|_{0, \Omega} + \|\ell\|_{0, \Omega} + \|\buu_d\|_{0, \Omega} \right)^2.
	\end{aligned}
\end{equation}
\end{theorem}
\begin{proof}
We sketch the proof following the Theorem~\ref{thm:abs-control}.
 The optimality conditions \eqref{eq:optimality} and \eqref{eq:optimality-h} for any $\bx_h \in \bY_{d,h}$ and Lemma~\ref{lem:est-bound-h}
    gives the bound for the estimates of control as
	\begin{equation}\label{est1-y}
	\begin{aligned} 
		\gamma \| \by - \by_h \|_{0, \Omega}^2 \le  - \sum_{K \in \cT_h}  \Big( & ( \bPi^{0}_K (\tP_h^u \bz - \bz_h ), \by - \by_h )_{0, K} + (\bz - \bPi^{0}_K (\tP_h^u \bz), \by - \by_h )_{0, K} \\
		& +  ( \bPi^{0}_K \bz_h + \gamma \by_h, \by - \bx_h )_{0, K} \Big) \\
        := T_1 +T_2+T_3. &
	\end{aligned}
\end{equation}
In order to bound the term $T_1$, we approach to the discrete problems \eqref{eq:weak-h-D}, \eqref{eq:Poro-proj-D} and \eqref{eq:Poro-adjoint-D}. We achieve the following error equation, and
	adding the equations obtained via testing the functions with $\bv_h =\tP_h^u \bz - \bz_h, \,q_h =\tP_h^p p - p_h$ and $\phi_h=\tP_h^{\psi} \psi - \psi_h$
    , and $\bv_h =P_h^u \bu - \bu_h$, $q_h =P_h^p p - p_h$, $\phi_h =P_h^{\psi} \psi - \psi_h$ 
    , yield
		\begin{align*} 
			T_1 & = \sum_{K \in \cT_h} \Big( ( \bu - \bu_h, \bPi^{0}_K (P_h^u \bu - \bu_h) )_{0,K} +( p- p_h, \Pi^{0}_K (P_h^p p - p_h) )_{0,K} \\
			& \qquad \qquad + ( \psi- \psi_h, P_h^\psi \psi - \psi_h )_{0,K} \Big) \\
			& = \sum_{K \in \cT_h}  \Big( \| \bPi^{0}_K (P_h^u \bu - \bu_h)\|_{0,K}^2 + \| \Pi^{0}_K (P_h^p p - p_h)\|_{0,K}^2 + \| P_h^\psi \psi - \psi_h \|_{0,K}^2 \\
			& \qquad \qquad + ( \bu - P_h^u \bu, \bPi^{0}_K (P_h^u \bu - \bu_h) )_{0,K} +( p- P_h^p p, \Pi^{0}_K (P_h^p p - p_h) )_{0,K} \\
			& \qquad \qquad + ( \psi- P_h^\psi \psi, P_h^\psi \psi - \psi_h )_{0,K}  \Big)
		\end{align*}
    With the help of projection $\bPi^Y$, the term $T_3$ can be rewritten on each $K \in \cT_h$ as
\begin{align*}
	& 
    T_3= ( \bPi^{0}_K \bz_h + \gamma \by_h, \by - \bPi^Y \by )_{0, K}  \nonumber\\
	& = ( \bPi^{0}_K (\bz_h - \bz )+ \gamma (\by_h - \bPi^Y \by), \by - \bPi^Y \by )_{0, K}+  ( \bPi^{0}_K \bz+ \gamma \bPi^Y \by, \by - \bPi^Y \by )_{0, K} \nonumber\\
	& = ( \bPi^{0}_K (\bz_h - \bz)+ \gamma (\bPi^Y \by - \by), \by - \bPi^Y \by )_{0, K} +  ((\bI - \bPi^Y) ( \bPi^{0}_K \bz+ \gamma  \by), \by - \bPi^Y \by )_{0, K}. 
\end{align*}
    Thus the inequality \eqref{est1-y} with the use of orthogonality property of $L^2$ projection, we have, for $\bx_h = \bPi^Y \by $,
	\begin{align*}
		\gamma \| \by - \by_h \|_{0, \Omega}^2  
		\le -\sum_{K \in \cT_h} & \Big( \| \bPi^{0}_K (P_h^u \bu - \bu_h)\|_{0,K}^2 + \| \Pi^{0}_K (P_h^p p - p_h)\|_{0,K}^2 + \| P_h^\psi \psi - \psi_h \|_{0,K}^2 \\
		& + ( \bu - P_h^u \bu, \bPi^{0}_K (P_h^u \bu - \bu_h) )_{0,K} +( p- P_h^p p, \Pi^{0}_K (P_h^p p - p_h) )_{0,K} \\
		& + ( \psi- P_h^\psi \psi, P_h^\psi \psi - \psi_h )_{0,K} + (\bz - \bPi^{0}_K (\tP_h^u \bz), \by - \by_h )_{0, K}  \\
		& \, + ( \bPi^{0}_K (\bz_h- \bz) + \gamma (\bPi^Y \by - \by), \by - \bPi^Y \by )_{0, K} \\
		& \, + ( (\bI - \bPi^Y)(\bPi^{0}_K \bz + \gamma \by), \by - \bPi^Y \by )_{0, K}\Big).
	\end{align*}
	We can conclude with the help of Holder's inequality, manipulating the terms and Young's inequality,
	\begin{align*}
		\gamma \| \by - \by_h \|_{0, \Omega}^2 &+ \sum_{K \in \cT_h}  
			\Big( \| \bPi^{0}_K (P_h^u \bu - \bu_h)\|_{0,K}^2 + \| \Pi^{0}_K (P_h^p p - p_h)\|_{0,K}^2  + \| P_h^\psi \psi - \psi_h \|_{0,K}^2\Big) \\
		\lesssim & \| \bu - P_h^u \bu\|_{0,\Omega}^2 +\| p- P_h^p p\|_{0,\Omega}^2 + \|  \psi- P_h^\psi \psi \|_{0, \Omega}^2 + \frac{1}{\gamma} \|\bz - \tP_h^u \bz\|_{0,\Omega}^2 + \| \bz - \bz_h\|_{0, \Omega}^2 \\
		& + \sum_{K \in \cT_h} \Big(  \frac{1}{\gamma} \|\bz - \bPi^{0}_K \bz \|_{0,K}^2  + \|(\bI - \bPi^Y) \bPi^{0}_K \bz\|_{0, K}^2 \Big) + \gamma\| (\bI - \bPi^Y) \by\|_{0,\Omega}^2 .
	\end{align*}
	Here the estimates of $\| \bz - \bz_h\|_{0, \Omega}$ can be achieved
    using the projection estimate and conclude the proof.
    
\end{proof}

\bigskip In the next corollary, we seek the estimates of state and costate variables in energy norm.
\begin{corollary}
	Assume $(\buu, \by, \buz) \in \bbV \times \bYd \times \bbV$ is solution of problem \eqref{eq:weak-D} and $(\buu_h, \by_h, \buz_h) \in \bbV_h \times \bY_{d,h} \times \bbV_h$ is solution of problem \eqref{eq:weak-h-D}. Then we have the estimate bounds: for $s \in (0, 1]$,
	\begin{align*}
		 2 \mu \| \beps (\bu - \bu_h)\|_{0, \Omega}^2  + \left( c_0 + \frac{\alpha^2}{\lambda} \right) \| p - p_h \|_{0, \Omega}^2   & + \frac{\kappa_1}{\eta} \| \nabla ( p - p_h)\|_{0, \Omega}^2 + \| \psi - \psi_h \|_{0, \Omega}^2   \\
		+ 2 \mu \| \beps(\bz - \bz_h)\|_{0, \Omega}^2 & \le C h^{2 s} \left(|\by|_{s, \Omega}+ 
        \|\ff\|_{0, \Omega} + \|\ell\|_{0, \Omega} + \|\buu_d\|_{0, \Omega} \right)^2.
	\end{align*}
\end{corollary}
	\begin{proof}
		Taking the test functions $\bv_h =P_h^u \bu - \bu_h$, $q_h =P_h^p p - p_h$, $\phi_h =P_h^{\psi} \psi - \psi_h$ in error equation
        gives
		\begin{equation*}
			\begin{aligned}
				a_{1,h}(P_h^u \bu - \bu_h, P_h^u \bu - \bu_h) +  b_1(P_h^u \bu - \bu_h, P_h^{\psi} \psi - \psi_h)     = \sum_{K \in \cT_h} ( \by - \by_h, \bPi^{0}_K (P_h^u \bu - \bu_h) )_{0,K},\\
				 a_{2,h}(P_h^p p - p_h, P_h^p p - p_h)   -    b_2(P_h^p p - p_h, P_h^{\psi} \psi - \psi_h )  = 0, \qquad \qquad \\
				b_1(P_h^u \bu - \bu_h, P_h^{\psi} \psi - \psi_h)    +  b_2(P_h^p p - p_h, P_h^{\psi} \psi - \psi_h)- a_3(P_h^\psi \psi - \psi_h, P_h^{\psi} \psi - \psi_h)  =0. \qquad  \qquad
			\end{aligned}
		\end{equation*}
		Subtracting the last two equations from first equation in the above system imply
		\begin{align*}
			a_{1,h}(P_h^u \bu - \bu_h, P_h^u \bu - \bu_h)  + a_{2,h}(P_h^p p - p_h, P_h^p p - p_h)  -  2  \, b_2(P_h^p p - p_h, P_h^{\psi} \psi - \psi_h ) \\
			+ a_3(P_h^\psi \psi - \psi_h, P_h^{\psi} \psi - \psi_h) = \sum_{K \in \cT_h} ( \by - \by_h, \bPi^{0}_K (P_h^u \bu - \bu_h) )_{0,K}.
		\end{align*}
		Using the coercivity of discrete bilinear forms and Korn's inequality, 
        we get the bound in terms of thes estimates of control variable as,
		\begin{align*}
			\frac{\mu}{2} \| \beps( P_h^u \bu - \bu_h) \|_{0, \Omega}^2  + c_0 \| P_h^p p - p_h \|_{0, \Omega}^2   + \frac{\kappa_1}{\eta} & \| \nabla (P_h^p p - p_h)\|_{0, \Omega}^2 + \frac{\alpha^2}{\lambda} \sum_{K \in \cT_h} \|(I - \Pi^{0}_K) (P_h^p p - p_h) \|_{0,K}^2  \\
			+ \frac{1}{\lambda} \sum_{K \in \cT_h} \| \alpha  \Pi^{0}_K (P_h^p p - p_h) - (P_h^\psi \psi - \psi_h) \|_{0,K}^2 & \le \frac{C_K^2}{2\mu} \| \by - \by_h\|_{0, \Omega}^2.
		\end{align*}
		In addition, the discrete inf-sup condition of $b(\cdot, \cdot)$,
triangle's inequality, and Lemma~\ref{lem:proj} imply
		\begin{align*}
			\mu \| \beps(\bu - \bu_h) \|_{0, \Omega}^2  + & \left(c_0 + \frac{\alpha^2}{\lambda} \right) \| p - p_h \|_{0, \Omega}^2   + \frac{\kappa_1}{\eta} \| \nabla (p - p_h)\|_{0, \Omega}^2 +  \|\psi - \psi_h\|_{0, \Omega}^2   \\
			& \lesssim \left( \frac{C_K^2}{2\mu} + \frac{C^2 +1}{\tilde{\beta}^2} \right)\| \by - \by_h\|_{0, \Omega}^2 + \frac{\mu ^2}{\tilde{\beta}^2} \|\bu - P_h^u \bu\|_{0, \Omega}^2.
		\end{align*}
		Next, we proceed for the estimates of adjoint states by taking the test functions $\bv_h =\tP_h^u \bz - \bz_h, q_h =\tP_h^p p - p_h$ and $\phi_h=\tP_h^{\psi} \psi - \psi_h$ in error equation
        in same manner, and get
		\begin{align*}
			a_{1,h}(\tP_h^u \bz - \bz_h, \tP_h^u \bz - \bz_h) +  b_1(\tP_h^u \bz - \bz_h, \tP_h^{\psi} \zeta - \zeta_h)     =\; \sum_{K \in \cT_h} ( \bu - \bu_h, \bPi^{0}_K (\tP_h^u \bz - \bz_h) )_{0,K},\\
			a_{2,h}(\tP_h^p r - r_h, \tP_h^p r - r_h)   +    b_2(q\tP_h^p r - r_h, \tP_h^{\psi} \zeta - \zeta_h)  =\sum_{K \in \cT_h} ( p- p_h, \Pi^{0}_K (\tP_h^p r - r_h) )_{0,K},  \\
			b_1(\tP_h^u \bz - \bz_h, \tP_h^\psi \zeta - \zeta_h)  -  b_2(\tP_h^p r - r_h, \tP_h^\psi \zeta - \zeta_h) - a_3(\tP_h^\psi \zeta - \zeta_h, \tP_h^\psi \zeta - \zeta_h)  =  \sum_{K \in \cT_h} ( \psi- \psi_h, \tP_h^\psi \zeta - \zeta_h )_{0,K}. \quad  		
		\end{align*}
		Then we have
		\begin{align*}
			\mu \| \beps( \tP_h^u \bz - \bz_h) \|_{0, \Omega}^2  + \left(c_0 + \frac{\alpha^2}{\lambda} \right) &\| \tP_h^p r - r_h \|_{0, \Omega}^2   + \frac{\kappa_1}{\eta} \| \nabla (\tP_h^p r - r_h)\|_{0, \Omega}^2 + \|\tP_h^\psi \zeta - \zeta_h\|_{0, \Omega}^2   \\
			& \lesssim \big( \| \bu - \bu_h\|_{0, \Omega}^2  + \| p - p_h\|_{0, \Omega}^2 + \| \psi - \psi_h\|_{0, \Omega}^2\big).
		\end{align*}
		Thus, the result from Lemmas~\ref{lem:proj}--\ref{lem:proj-adjoint} alongwith Theorem~\ref{thm:control-est} can conclude the proof.
	\end{proof}

\subsection{Discrete Neumann boundary control formulation}
%

We introduce the discrete space for the Neumann boundary control problem with the set of all Neumann boundary edges $\mathcal{E}_h^N$ as
    \begin{align*}
		\bY_h &:= \lbrace \by_h \in \bY: \by_h |_e \in [\bbP_0(e)]^2 \, \forall e \in  \mathcal{E}_h^N \rbrace, \\
        \bY_{d,h} & := \lbrace \by_h \in \bY_h: \by_a \le \by_h \le \by_b \rbrace.
	\end{align*}
    Then the virtual element formulation of the continuous Neumann boundary control problem \eqref{eq:weak-N} is defined as: seek $(\buu_h, \by_h, \buz_h) \in \bbV_h \times \bY_{d,h} \times \bbV_h$
	such that
	\begin{subequations} \label{eq:weak-h-N}
		\begin{align}
			a_{1,h}(\bu_h, \bv_h) + b_1(\bv_h, \psi_h) &=  \sum_{K \in \cT_h} ( \ff, \bPi^{0}_K \bv_h )_{0, K}  + \sum_{e \in \mathcal{E}_h^N} (\by_h, \bv_h )_{0, e}    & \forall \bv_h \in \bV_h, \label{eq:weak-h-N-a}\\
			a_{2,h}(p_h, q_h) - b_2(q_h, \psi_h) & = \sum_{K \in \cT_h} ( \ell, \Pi^{0}_K q_h )_{0, K}  & \forall q_h \in Q_h, \label{eq:weak-h-N-b}\\
			b_1(\bu_h, \phi_h) + b_2(p_h, \phi_h) - a_3(\psi_h, \phi_h) & =0  & \forall \phi_h \in Z_h, \label{eq:weak-h-N-c}\\
			a_{1,h}(\bv_h, \bz_h) + b_1(\bv_h, \zeta_h) &= \sum_{K \in \cT_h} ( \bu_h - \bu_d, \bPi^{0}_K \bv_h)_{0, K}   & \forall \bv_h \in \bV_h, \label{eq:weak-h-N-d} \\
			a_{2,h}(r_h, q_h) + b_2(q_h, \zeta_h) & = \sum_{K \in \cT_h} (p_h - p_d, \Pi^{0}_K  q_h)_{0, K}  & \forall q_h \in Q_h, \label{eq:weak-h-N-e} \\
			b_1(\bz_h, \phi_h) - b_2(r_h, \phi_h) - a_3(\zeta_h, \phi_h) & =\sum_{K \in \cT_h} (\psi_h - \psi_d,  \phi_h)_{0, K} \,\, & \forall \phi_h \in Z_h, \label{eq:weak-h-N-f} \\
			\sum_{e \in \mathcal{E}_h^N} (\bz_h + \gamma \by_h, \bx_h - \by_h )_{0,e} & \ge  0 &\quad \forall \bx_h \in \bY_{d,h}.  \label{eq:optimality-h-N}
		\end{align}
	\end{subequations}

    \noindent Similar to the distributed control optimality condition, we can rewrite the optimality condition \eqref{eq:optimality-h-N} as $\by_h|_e:= \bPi_{[a,b]} \left((-1/\gamma) \bPi^{0,0}_K \bz_h|_e  \right)$ for each $e \in \mathcal{E}_h^N$.

    \subsubsection{A priori error estimates}
    In order to develop the error estimates, we assume that the continuous solution of the problem have the following regularity on the polygonal domain $\Omega$ for $0 < s \le 1$ 
    \begin{gather} \label{reg}
        \bu, \bz \in [H^{1+s}(\Omega)]^2, \,\, p, r \in H^{1+s}(\Omega), \,\,
        \phi, \chi \in H^s(\Omega) \,\, \text{ and } \,\, \by \in [H^{\frac{1}{2} +s }(\Gamma_N)]^2 
    \end{gather}

    \begin{lemma}
        Assuming the regularity of the solution \eqref{reg}, we have the estimate of the projection $\bPi^{Y}$ for a constant $C$ independent of $h$ as, for $s \in (0,1]$
        \begin{align*}
            \| \bz - \bPi^{Y} \bz \|_{0, \Gamma} \le C h^{\frac{1}{2} + s} \| \bz \|_{1+s, \Omega}, \qquad \text{and} \qquad 
            \| \by - \bPi^{Y} \by \|_{0, \Gamma_N} \le C h^{\frac{1}{2} + s} \| \by \|_{\frac{1}{2} + s, \Gamma_N}.
        \end{align*}
    \end{lemma}
    \noindent Now, we define a poroelastic projection operator for state and adjoint variable in case of Neumann boundary control problem as follows,
	\begin{itemize}
	\item {\bf Poroelastic projection:} Define $P_h=(P_h^u, P_h^p, P_h^{\psi}): \bbV \to \bbV_h$ as
	\begin{subequations}\label{eq:Poro-proj-N}
		\begin{alignat}{6}
			a_{1,h}(P_h^u \bu, \bv_h)+  b_1(\bv_h, P_h^{\psi} \psi)     &=& \sum_{K \in \cT_h} (\ff, \bPi^{0}_K \bv_h )_{0,K} + \sum_{e \in \mathcal{E}_h^N} (\by, \bv_h )_{0,e}   &&\forall \bv_h \in \bV_h,\\
			a_{2,h}(P_h^p p, q_h)   -    b_2(q_h, P_h^{\psi} \psi)  &=& \sum_{K \in \cT_h} (\ell, \Pi^{0}_K q_h )_{0,K}  \quad &&\forall q_h \in Q_h,  \\
			b_1(P_h^u \bu, \phi_h)  +  b_2(P_h^p p, \phi_h) - a_3(P_h^\psi \psi, \phi_h)  &=&0 \qquad     \qquad &&\forall \phi_h \in Z_h.
	\end{alignat}\end{subequations}

    \item {\bf Poroelastic projection for adjoint problem:} Define $\tP_h=(\tP_h^u, \tP_h^p, \tP_h^{\psi}): \bbV \to \bbV_h$ as
	\begin{subequations}\label{eq:Poro-adjoint-N}
		\begin{alignat}{6}
			a_{1,h}(\bv_h, \tP_h^u \bz) +  b_1(\bv_h, \tP_h^{\psi} \zeta)     &=&\; \sum_{K \in \cT_h} ( \bu - \bu_d, \bPi^{0}_K \bv_h )_{0,K}  &\quad&\forall \bv_h \in \bV_h,\\
			a_{2,h}(q_h, \tP_h^p r)   +    b_2(q_h, \tP_h^{\psi} \zeta)  &=& \sum_{K \in \cT_h} ( p- p_d, \Pi^{0}_K q_h )_{0,K}  &\quad&\forall q_h \in Q_h,  \\
			b_1(\tP_h^u \bz, \phi_h)    -  b_2(\tP_h^p r, \phi_h)- a_3(\tP_h^\psi \zeta, \phi_h)  &=& \sum_{K \in \cT_h} ( \psi- \psi_d, \phi_h )_{0,K} \quad    &\quad&\forall \phi_h \in Z_h.
	\end{alignat}\end{subequations}
    \end{itemize}
    
    Following the proof of Lemma~\ref{lem:proj} and Lemma~\ref{lem:dis-proj-t}, the error estimates of the poroelastic projections for state and adjoint variables are stated in the next two lemmas for the Neumann boundary control.
    \begin{lemma} \label{lem:proj}
		Assume that $\buu \in \bbV$ and $P_h \buu \in \bbV_h$ are the solution of continuous problem  \eqref{eq:weak-N-a}-\eqref{eq:weak-N-c} and poroelastic projection \eqref{eq:Poro-proj-N} respectively.
		 Then it holds, for constant $C$ independent of $h$,
		\begin{equation}\label{est:proj-D}
			\begin{aligned} 
			2 \mu \| \beps(\bu - P_h^u \bu) \|_{0, \Omega}^2  + c_0 \| p - P_h^p p \|_{0, \Omega}^2 + & \frac{\kappa_1}{\eta} \| \nabla (p - P_h^p p)\|_{0, \Omega}^2+ \| \psi - P_h^{\psi} \psi \|_{0, \Omega}^2 \\
			& \le C h^{2s} ( \|\ff \|_{0, \Omega}^2 + \|\by \|_{0, \Gamma_N}^2 + \|\ell\|_{0, \Omega}^2  ).
		\end{aligned}
	\end{equation}
	for $s \in (0,1]$.
	\end{lemma}

\begin{lemma} \label{lem:proj-adjoint}
	Assume that $\buz \in \bbV$ and $\tP_h \buz \in \bbV_h$ satisfy the equations \eqref{eq:weak-N-d}-\eqref{eq:weak-N-f} and \eqref{eq:Poro-adjoint-N} respectively, then it holds for constant $C$, independent of $h$, and for $s \in (0, 1]$ as
	 \begin{equation}\label{est:adjointproj-N}
	 	\begin{aligned} 
	 		2 \mu \| \bz - \tP_h^u \bz \|_{0, \Omega}^2 + c_0 \| r - \tP_h^p r \|_{0, \Omega}^2 & + \frac{\kappa_1}{\eta} \| \nabla (r - \tP_h^p r )\|_{0, \Omega}^2+ \| \zeta - \tP_h^{\psi} \zeta \|_{0, \Omega}^2 \\
	 		&\qquad \le C h^{2s} (
            \|\by \|_{0, \Gamma_N}^2 + \|\ff \|_{0, \Omega}^2 + \|\ell\|_{0, \Omega}^2 + \| \buu_d\|_{0, \Omega}^2 ).
	 	\end{aligned}
	  \end{equation}
	\end{lemma}


\noindent We follow the estimates from Theorem~\ref{thm:abs-control} and Corollary~\ref{cor:abs-state} and achieve the following result.
\begin{theorem} \label{thm:abs-control-N}
	Assume that $(\buu,\by,\buz) \in \bbV \times \bYd \times \bbV$ and $(\buu_h, \by_h, \buz_h) \in \bbV_h \times \bY_{d,h} \times \bbV_h$ are the solutions of equations \eqref{eq:weak-N} and \eqref{eq:weak-h-N} respectively. The following estimate holds for $0 < s \le 1$,
	\begin{align*}
		\sum_{e \in \mathcal{E}_h^N} \| \by - \by_h \|_{0,e}^2 & 
        + \| \buu - \buu_h \|_{\bbV}^2 
        + \| \buz - \buz_h \|_{\bbV}^2 \le C h^{2s} (h^{\frac{1}{2}} |\by|_{s+\frac{1}{2}, \Gamma_N} 
        + \|\ff \|_{0, \Omega} + \|\ell\|_{0, \Omega})^2. 
	\end{align*}
\end{theorem}
The proof of this theorem follows exactly from abstract results in Theorem~\ref{thm:abs-control} and using the following bounds:
	\begin{align*}
		\sum_{e \in \mathcal{E}_h^N} \big( \bz_h - \tP_h^u \bz, \by - \by_h \big)_{0,e} \ge  \sum_{e \in \mathcal{E}_h^N} \Big( \gamma \| \by - \by_h \|_{0,e}^2 + \big( \bz - \tP_h^u \bz, \by - \by_h \big)_{0,e} +  \big( \bz_h + \gamma \by_h, \by - \bx_h \big)_{0,e} \Big),
	\end{align*}
for any $\bx_h \in \bY_{d,h}$.

	\section{Numerical Experiments} \label{sec:numer}

    In this section, we will present a few numerical tests through the PDAS strategy \cite{troltzsch10} for solving the finite-dimensional optimization problem \eqref{eq:weak-h-D}. 

    For construction of the stiffness matrix, the discrete solutions are expanded in terms of basis functions associated with the chosen VEM spaces for displacement, pressure, total pressure, and control. 
    Say $\xi_i ( 1 \le i \le N^{\bV})$,
	$\chi_j (1 \le j \le N^Q)$, 		
	$\Psi_l ( 1 \le l \le N^Z$, where $N^Z$ coincides with the number of elements in $\mathcal{T}_h$) and $e_l (1 \le l \le N^Y)$ are the basis functions for the spaces $\bV_h, Q_h, Z_h$ and $\bY_{d,h}$ respectively. 
    Writing the discrete solution as linear combinations of these basis functions as,
        \begin{gather*}
            \bu_h (\bx) := \sum_{i=1}^{N^{\bV}} \mathcal{U}(i) \, \xi_i(\bx), \,\, p_h (\bx) := \sum_{j=1}^{N^Q} \mathcal{P}(j)\,  \chi_j(\bx), \,\, \psi_h (\bx) := \sum_{l=1}^{N^Z} \mathcal{S}(l) \, \Psi_l(\bx),\,\,  \bz_h (\bx) := \sum_{i=1}^{N^{\bV}} \mathcal{Z}(i) \, \xi_i(\bx), \\
            r_h (\bx) := \sum_{j=1}^{N^Q} \mathcal{R}(j)\,  \chi_j(\bx), \quad 
            \zeta_h (\bx) := \sum_{l=1}^{N^Z} \mathcal{T}(l) \, \Psi_l(\bx), \, \text{ and } \,\, \by_h (\bx) :=  \sum_{l=1}^{N^Y} \mathcal{Y}(l) \, e_l(\bx).
        \end{gather*}
     We now assemble the discrete system into its global block structure as,
		\begin{equation} \label{compute:localmat}
        \begin{aligned}
		A=	\begin{bmatrix}
				A1 & \cero & B1 \\
				\cero & A2 & -B2 \\
				B1^T & B2^T & -A3
			\end{bmatrix}, \,\,
		\vec{\mathcal{U}}=
			\begin{bmatrix}
				\mathcal{U} \\ 
				\mathcal{P} \\
				\mathcal{S} 
			\end{bmatrix}, &
			\,\,
		A^*=\begin{bmatrix}
				A1 & \cero & B1 \\
				\cero & A2 & B2 \\
				B1^T & -B2^T & -A3
			\end{bmatrix},
			\, \,
		\vec{\mathcal{Z}}=
			\begin{bmatrix}
				\mathcal{Z} \\ 
				\mathcal{R} \\
				\mathcal{T}
			\end{bmatrix} \\
        M=\begin{bmatrix}
				M1 & \cero & \cero \\
				\cero & M2 & \cero \\
				\cero & \cero & M3
			\end{bmatrix}, &
		\quad \text{and } \quad
		{\mu}:= - \big( {\mathcal{Y}} + \gamma^{-1} \mathcal{Z} \big).
        \end{aligned}
		\end{equation}
	where local matrices for the respective bilinear forms are computed as,
	\begin{align*}
		{A1}&:=	
			\begin{bmatrix}
				[a_{1,h}(\xi_i, \xi_r)]_{1 \le i,r \le N^{\bV}}
			\end{bmatrix}, \qquad  
        B1 :=	\begin{bmatrix}
			[b_1(\xi_r, \varphi_l)]_{1 \le r \le N^{\bV}, \, 1 \le l \le N_Z}  
			\end{bmatrix}, \\
		{A2}&:=	
			\begin{bmatrix}
				[a_{2,h}(\chi_i, \chi_r)]_{1 \le i,r \le N^{Q}}
			\end{bmatrix}, 
		\qquad B2 :=	\begin{bmatrix}
					[b_2(\chi_r, \varphi_l)]_{1 \le r \le N^Q, \, 1 \le l \le N_Z}  
			\end{bmatrix}, 	\\
		{A3}&:=	\begin{bmatrix}
					[a_{3}(\varphi_i, \varphi_r)]_{1 \le i,r \le N^{Z}}
				\end{bmatrix}, 
		\qquad  M1 := \begin{bmatrix}
				[(\bPi^{0}_K \xi_i, \bPi^{0}_K\xi_r)_{0,K}]_{1 \le i,r \le N^{\bV}}
			\end{bmatrix}, \\
        M2 & := \begin{bmatrix}
				[(\Pi^{0}_K \chi_i, \Pi^{0}_K\chi_r)_{0,K}]_{1 \le i,r \le N^{Q}}
			\end{bmatrix}, \qquad 
        M3  := \begin{bmatrix}
				[(\Psi_i, \Psi_r)_{0,K}]_{1 \le i,r \le N^{Z}}
			\end{bmatrix}, \\
        \vec{\mathcal{F}} &:=
			\begin{bmatrix} 
            [\sum_{K \in \cT_h} (\ff, \bPi^{0}_K \xi_r)_{0,K} ]_{1 \le r \le N^{\bV}} \\
			[ \sum_{K \in \cT_h}\left( \ell, \Pi^0_K \chi_j  \right)_{0,K}]_{1 \le j \le N^Q} \\ 
			[ \cero_j]_{1 \le j \le N^Z}
			\end{bmatrix},
            \qquad
        \vec{\mathcal{U}}_d :=
			\begin{bmatrix} 
            [ \sum_{K \in \cT_h}\left( \bu_d, \bPi^0_K \xi_j  \right)_{0,K}]_{1 \le j \le N^{\bV}} \\
			[ \sum_{K \in \cT_h}\left( p_d, \Pi^0_K \chi_j  \right)_{0,K}]_{1 \le j \le N^Q} \\ 
			[ \sum_{K \in \cT_h}\left( \psi_d, \Pi^0_K \varphi_j  \right)_{0,K}]_{1 \le j \le N^Z}
			\end{bmatrix} \\
        B &:= \begin{bmatrix}
				[ \langle B_h \xi_i, e_l \rangle_{\bY}]_{1 \le i \le N^{\bV}, \, 1 \le l \le N^Y} \\
                \cero_{1 \le i \le N^{Q} + N^Y}
			\end{bmatrix},  \quad
            \text{ and } \quad
        D := \begin{bmatrix}
				[ (e_l, e_m)_{0,K}]_{ 1 \le l,m \le N^Y}
			\end{bmatrix}.
	\end{align*}

    \bigskip
    
    \bigskip
    
    \noindent
    We present below the algorithmic procedure followed for conducting the numerical tests. The finite‑dimensional optimization problem is solved using the PDAS strategy \cite{troltzsch10}, which iteratively updates the active set by checking the control constraints against the projection formula, while solving the state and adjoint equations on the inactive set. 
    
    \begin{algorithm*} \label{alg}
	\caption{Implementation} 
    {
	\begin{algorithmic}[1]
    \State Given tolerance $\text{tol}$ and number of iterations $N$
    \State Compute load vectors $\vec{\mathcal{F}}$, data vector $\vec{\mathcal{U}_d}$, and local matrices $A, A^*, B, M$ defined in \eqref{compute:localmat}
    \While {$\text{error}<\text{tol}$}
		\For {$n=1,2,\ldots,N$}
            \State Start with $\mathcal{Y}^0$, $\vec{\mathcal{U}}^0$, $\vec{\mathcal{Z}}^0$ as zero vector
            \State Find Active sets: for $\mu^{n-1}$ in \eqref{compute:localmat},
            \begin{align*}
            A_n^a = \{ i \in \{ 1, \dots, N^Y \}: {\mathcal{Y}}^{n-1}(i) + 
            {\mu}^{n-1}(i) < y_a  \} \\
            A_n^b = \{ i \in \{ 1, \dots, N^Y \}: {\mathcal{Y}}^{n-1}(i) + 
            {\mu}^{n-1}(i) > y_b  \}
            \end{align*}
            \State Compute matrices cooresponding to the active sets:
            \begin{align*}
                X_a(i,i)= 
                \begin{cases}
                1 \, \text{ if } i \in A_n^a \\
                0 \quad \text{elsewhere}
                \end{cases}, 
                \quad \text{and} \quad \mathcal{Y}_a(i)=y_a \quad \forall \,1\le i\le N^Y.
            \end{align*}
            \State Compute local matrices:
            \begin{gather*}
            C  := (\gamma \, D)^{-1} (I - X_a - X_b) B^T, \quad  \text{and} \quad
            {\mathcal{Y}}_r :=X_a {\mathcal{Y}}_a +X_b {\mathcal{Y}}_b,
            \end{gather*}
            \State {Solve for $\vec{\mathcal{U}}^n $, $\vec{\mathcal{Z}}^n$ and ${\mathcal{Y}}^n$ for $n\ge 1$ such that}
            \begin{align*}
			\begin{bmatrix}
				A & \cero & -B \\
				-M & A^* & \cero \\
				\cero & C & I
			\end{bmatrix}
			\begin{bmatrix}
				\vec{\mathcal{U}}^n \\
				\vec{\mathcal{Z}}^n \\
				{\mathcal{Y}}^n
			\end{bmatrix}
			= \begin{bmatrix}
				\vec{\mathcal{F}}  \\
				-\vec{\mathcal{U}_d} \\
				{\mathcal{Y}_r}
			\end{bmatrix} 
            \end{align*}
            \State Evaluate error for step $n$ with exact solution 
		\EndFor
    \EndWhile
	\end{algorithmic} 
    }
\end{algorithm*}

        Denote the number of elements in the mesh by $N$, the mesh-size for the refinement level $i=1,2, \dots, 6$ is denoted by $h_i$, and the $L^2, \,H^1-$ error associated with the variables are denoted as $e_0^{i}(\cdot), \, e_1^i(\cdot)$ respectively.
        The rates of convergence $r_0^i, \, r_1^i, \, i=2,3, \dots, 6$ are computed as 
        \[r_0^i(\cdot) = \frac{\log(e_0^{i}(\cdot))- \log(e_{0}^{i-1}(\cdot))}{\log(h_i)- \log(h_{i-1})}, \quad r_1^i(\cdot) = \frac{\log(e_1^{i}(\cdot))- \log(e_{1}^{i-1}(\cdot))}{\log(h_i)- \log(h_{i-1})}
        . \]

        \begin{figure}[t!]
            \centering
            \includegraphics[width=0.45\linewidth]{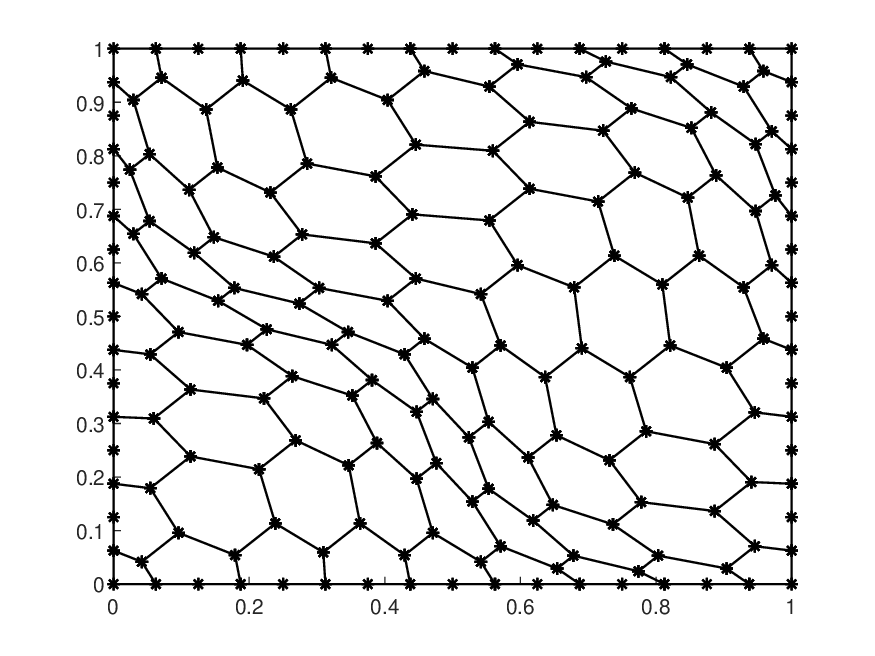}
            \includegraphics[width=0.45\linewidth]{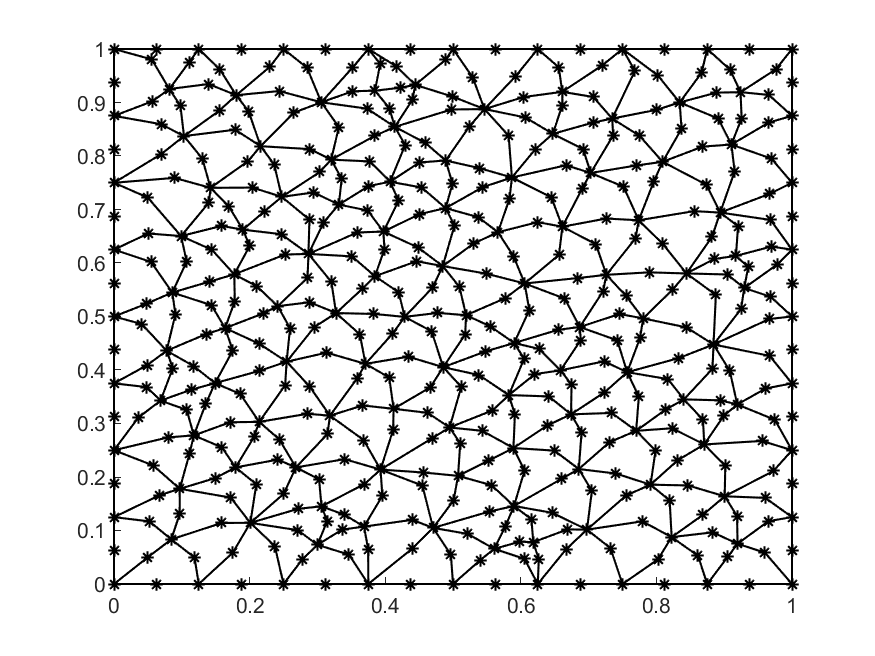}
            \caption{Distorted hexagonal mesh(left) with  $h=0.046$ and web-like mesh(right) with $h=0.0445$.}
            \label{fig:webmesh}
        \end{figure}

		\subsection{Convergence test with Distributed control} \label{numer:test1}
		
		We consider a manufactured example on unit square domain $\Omega:= (0,1)^2$ and discretized with the web-like mesh given in Figure~\ref{fig:webmesh}.
        The exact solution of displacement state, adjoint of displacement, pressure state, and adjoint of pressure are given as, 
		\begin{align*}
			\bu = \bz=
			\begin{pmatrix}
				- \cos 2 \pi x \,\sin 2 \pi y+ \sin 2\pi y+\sin^2 2 \pi x \sin^2 2\pi y \\
				\sin 2 \pi x \, \cos 2 \pi y- \sin 2 \pi x
			\end{pmatrix}, \quad
			p=r= \sin^2\pi x \, \sin^2\pi y.
		\end{align*}
		The physical parameters are chosen as
		\begin{align*}
			\nu=0.4, \, \text{Ec}=100, \, \kappa=1, \, \alpha=1, \, c_0=0.001,\, \eta=0.1, \, \lambda=\frac{\nu \, \text{Ec}}{(1+ \nu) (1-2 \, \nu)}, \, \text{ and } 		\mu=\frac{\text{Ec}}{2+2 \, \nu}.
		\end{align*}
        The Dirichlet boundary conditions are imposed for displacement and the Neumann boundary conditions for pressure on the entire boundary $\Gamma_D = \partial \Omega$ (that is $\Gamma_N = \emptyset$). The problem is then solved by constructing the load functions $\ff, \ell$, total pressure state $\psi$, adjoint of total pressure $\chi$, desired data $\bu_d, \, p_d, \, \psi_d$, and control $\by$ directly from the manufactured solution $\bu, \bz, p,  r$ and the parameters in the optimality system \eqref{eq:optimal-abstract}. We set the tolerance as $1e-06$, the maximum number of PDAS iteration as $20$, choose the penalty parameter for the cost functional as $\gamma=0.1$, and the admissible control space $\bY_d$ with $y_a=-1$ and $y_b=8.$ The discrete scheme is then implemented by following Algorithm~\ref{alg}, and the resulting numerical errors are reported in Table~\ref{table-Ex1}.
		
		\begin{table}[ht!]
			\setlength{\tabcolsep}{4.5pt}
			\centering 
			\caption{Computed errors for approximated state, adjoint and control solutions on distorted hexagonal mesh.}
			{\small
					\begin{tabular}{|rc|cccc|cccc|cc|cc|} 
					\hline\hline
					N   &   $h$  & $\mathrm{e}_0(\bu)$  &   $r_0(\bu)$   &   $\mathrm{e}_1(\bu)$  &   $r_1(\bu)$  & $\mathrm{e}_0(p)$  &   $r_0(p)$   &   $\mathrm{e}_1(p)$  &   $r_1(p)$ & $\mathrm{e}_0(\psi)$  &   $r_0(\psi)$   &   $\mathrm{e}_0(\by)$  &   $r_0(\by)$ \\
					\hline 
468 & 0.0462 & 0.4359   &  $\star$ & 0.7583  &   $\star$ & 0.15955 &   $\star$ & 0.5084 &   $\star$ & 0.7202  &   $\star$ & 0.9877 &   $\star$ \\
1636 & 0.0247 & 0.1302   &  1.74  &  0.3362  & 1.17  & 0.04914 &   1.70  &  0.2752   & 0.89 &  0.3173   & 1.18 &  0.8586  & 0.20 \\
6084 & 0.0128 & 0.0361  & 1.85 &  0.1598   &  1.07 & 0.01383 &   1.83  &  0.1431  &  0.94 &  0.1502  &  1.08 &  0.5065 &  0.76  \\
23428 & 0.0065 & 0.0094  &  1.93  &  0.0788 &  1.02 &  0.00363 & 1.93 & 0.07289  &  0.97 &  0.07402  &  1.02 &  0.19985  &  1.34 \\
91908 & 0.0033 & 0.0024 & 1.97  &  0.0393  &  1.00 & 0.00092 &  1.97 & 0.03679 &  0.99 &  0.03697 &  1.00  &  0.06863 &  1.54 \\
364036 & 0.0017 & 0.00061 & 1.99  &  0.0197  &  1.00 & 0.00023 &  1.99 & 0.01847 &  0.99 &  0.01851 &  1.00 &  0.02685 &  1.35 \\
					\hline
					\hline
			\end{tabular}}
			\smallskip			
			\label{table-Ex1}
		\end{table}

        In Table~\ref{table-Ex1}, we observed the optimal rate of convergence for the state, adjoint and control variables. Thus, we have verified the established theoretical results for distributed control problem for this manufactured example. The postprocessed control solution is illustrated in Figure~\ref{fig:control}, 
        \begin{figure}[h!]
            \centering
            \includegraphics[width=0.61\linewidth]
            {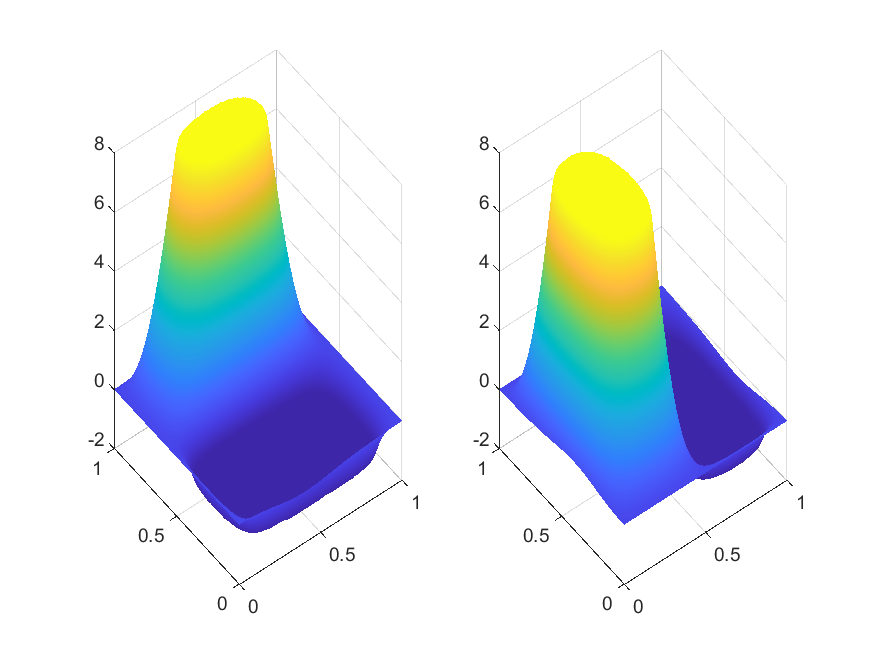}
            \caption{First and second components of post-processed plot of the computed piecewise constant control solution $\by_h$.}
            \label{fig:control}
        \end{figure}
        where the bounded control variable is projected onto the admissible set $\bYd$. The visualization highlights how the PDAS strategy enforces the control constraints while maintaining accuracy. The control field exhibits smooth behavior consistent with the prescribed manufactured solution, confirming both the robustness of the VEM discretization and the efficiency of the PDAS implementation.

	\subsection{Convergence test with Neumann control}
    \label{numer:test2}
    
    In this test, we take the domain $\Omega$ given in the test \ref{numer:test1}, where the boundary is partitioned into a non‑homogeneous Neumann boundary $\Gamma_N = [0,1] \times \{ 0\}$ and homogeneous Dirichlet boundary $\Gamma_D = \partial \Omega \backslash \Gamma_N$ for displacement state and adjoint variables. 
    The manufactured exact solutions are chosen to ensure smoothness and non‑trivial boundary behavior, and are prescribed as,
    \begin{align*}
			\bu = \bz=
			\begin{pmatrix}
				\cos x (1-x) \sin^2(\pi y) \\
				\cos x (1-x) \sin^2(\pi y)
			\end{pmatrix}, \quad
			p=r= \sin^2\pi x \, \sin^2\pi y.
		\end{align*}
    These exact solutions naturally lead to construct the load functions $\ff, \ell$, data functions $\bu_d, p_d, \psi_d$ for computation. The physical parameters are set to
		\begin{align*}
			\nu=0.35, \, \, \text{Ec}=100, \, \kappa=1, \, \alpha=1, \, c_0=0.5,\, \eta=0.1, \, \lambda=\frac{\nu \, \text{Ec}}{(1+ \nu) (1-2 \, \nu)}, \, \text{ and } 		\mu=\frac{\text{Ec}}{2+2 \, \nu}.
		\end{align*}
    The admissible space for the control $\bYd$ is defined with the bound $y_a=-0.85$, $y_b=0.05$ and penalty parameter as $\gamma=1$, and thus, the exact control solution is $\by= \Pi_{[a,b]} \left( - \bz \right).$

        \begin{figure}[ht!]
            \centering
            \includegraphics[width=0.47\linewidth]
            {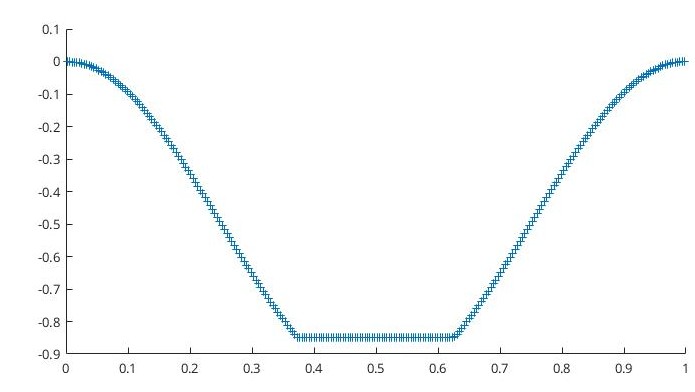}
            \qquad 
            \includegraphics[width=0.47\linewidth]
            {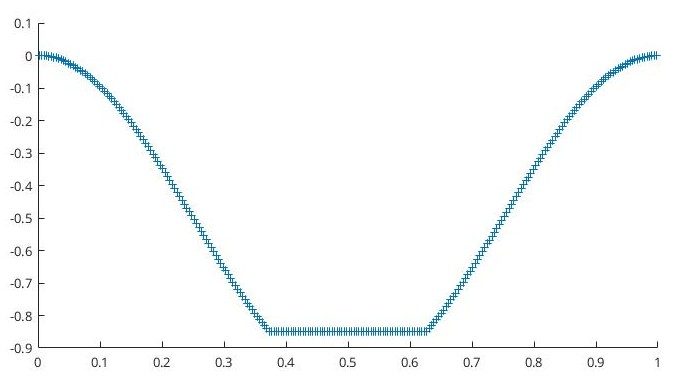}
            \caption{First and second components of plot of the computed piecewise constant control solution $\by_h$.}
            \label{fig:control_eg2}
        \end{figure}
        
        The computed control solution (shown on the left) and the exact control solution (shown on the right) are plotted along the Neumann boundary in Figure~\ref{fig:control_eg2}. This visual comparison highlights the accuracy of the numerical approximation in reproducing the prescribed boundary actuation. The close agreement between the two plots demonstrates that the discrete scheme captures the essential features of the control variable, thereby validating the theoretical convergence rates established for Neumann boundary control problems.


        \begin{table}[ht!]
			\setlength{\tabcolsep}{4.5pt}
			\centering 
			\caption{Computed errors of control variable and state variables.}
			{\small
					\begin{tabular}{|rc|cccc|cccc|cc|cc|} 
					\hline\hline
					N   &   $h$  & $\mathrm{e}_0(\bu)$  &   $r_0(\bu)$   &   $\mathrm{e}_1(\bu)$  &   $r_1(\bu)$  & $\mathrm{e}_0(p)$  &   $r_0(p)$   &   $\mathrm{e}_1(p)$  &   $r_1(p)$ & $\mathrm{e}_0(\psi)$  &   $r_0(\psi)$   &   $\mathrm{e}_0(\by)$  &   $r_0(\by)$ \\
					\hline 
312 & 0.0566 & 0.43999  &  $\star$ & 0.80689 &   $\star$ & 0.21523 &   $\star$ & 0.6463 &   $\star$ & 0.6988 &   $\star$ & 0.85105 &   $\star$ \\
944 & 0.0325 &  0.08071  &  2.45  &  0.35206   & 1.20  & 0.08408 &   1.36  &  0.3633   & 0.83 &  0.3191 &   1.13 &  0.83174  &  0.03 \\
4208 & 0.0154 & 0.01984  & 2.02 &  0.15529  &  1.18 & 0.01767 &   2.25  &  0.1635  &  1.15 &  0.1400 &  1.20 &  0.41757 & 0.99 
\\
17336 & 0.00759 & 0.00496  &  2.00  &  0.07844 &  0.99 &  0.00452 & 1.97 & 0.0820  &  1.00 &  0.0706 &  0.99 &  0.09143  &   2.19 \\
69592 & 0.00379 & 0.00117 & 2.08  &  0.03445  &  1.03 & 0.00112 &  2.02 & 0.0407 &  1.01 &  0.0345  & 1.03 &  0.02790 &  1.71 
\\
281344 & 0.00189 & 0.00030 & 1.99  &  0.01916  &  1.00 & 0.00028 & 1.98 & 0.0204 &  0.99 &  0.0171  & 1.01 & 0.00822 & 1.76 \\
					\hline
					\hline
			\end{tabular}}
			\smallskip			
			\label{table-Ex2}
		\end{table}
   	The computed errors for the state, adjoint, and control variables are reported in Table~\ref{table-Ex2}, where successive mesh refinements clearly demonstrate the expected convergence behavior. The numerical values confirm that the displacement state and adjoint variables achieve optimal accuracy under the imposed mixed boundary conditions, while the boundary control along $\Gamma_N$ converges consistently to the exact solution. The tabulated results not only validate the theoretical error estimates but also highlight the robustness of the scheme, showing that the control approximation remains stable and accurate without oscillations. This comprehensive error analysis reinforces the effectiveness of the proposed framework for Neumann boundary control problems.
    
    Compared with the distributed control case, this test emphasizes boundary‑localized actuation and demonstrates that the proposed framework is equally effective for both domain‑based and boundary‑based control formulations. Overall, the experiment establishes that the method is robust, stable, and capable of delivering optimal convergence rates for Neumann boundary control problems, thereby validating its applicability to practical engineering scenarios involving stress or flux control on boundaries.

			\section{Conclusion and future works}

            In this paper, we have developed a conforming virtual element scheme for optimal control problem (OCP) governed by three-field poroelasticity equation. The wellposedness of the three-field formulation of OCP is established and then developed the a priori error analysis of the discrete scheme. The numerical experiment verifies the optimal convergence rates for state, costate and control variables supporting the theoretical results. Further, the analysis can be extended for the time-dependent poroelasticity equation with distributed and Neumann boundary control problems.
            
            A natural extension of the present study is to investigate the virtual element approximations for the distributed and Neumann controls on pressure for poroelasticity equations. In addition, the future works will focus on extending the analysis to the numerical approximation for the Dirichlet control problems \cite{gudi25}. Unlike distributed or Neumann controls, Dirichlet controls impose constraints directly on boundary displacements, which introduces new analytical challenges such as handling non‑homogeneous boundary conditions and ensuring stability under low regularity assumptions.

			\bigskip
			\small 	
			\noindent\textbf{Acknowledgements.}
			The second author was supported by ANID-Chile through FONDECYT project 1261217 and
            by project Centro de Modelamiento Matem\'atico (CMM), FB210005, BASAL funds for centers of excellence. 
			Third author was supported by FONDECYT postdoctoral funding project 3240737 and by project Centro de Modelamiento Matem\'atico (CMM), FB210005, BASAL funds for centers of excellence.
			
			\small 	
			\noindent\textbf{Data Availability.} Enquiries about data availability should be directed to the authors.
			
			\small 	
			\noindent\textbf{Conflicts of interest.} The authors have no conflicts of interest to declare.
			

		\end{document}